\documentclass[10pt,a4paper,twoside, reqno, tbtags]{amsart}
\usepackage[utf8]{inputenc}
\usepackage[T1]{fontenc}
\usepackage{amsmath}
\usepackage{amssymb}
\usepackage{amscd}
\usepackage{mathrsfs}
\usepackage{enumerate}
\usepackage{verbatim}
\usepackage{caption}
\usepackage{graphicx}
\usepackage{amsfonts}
\usepackage{amsbsy}
\usepackage{eucal}
\usepackage{mathptmx}
\usepackage[numbers]{natbib}
\usepackage{hyperref}
\numberwithin{equation}{section}
\newtheorem{thm}{Theorem}[section]
\newtheorem{theorem}[thm]{Theorem}
\newtheorem{conjecture}[thm]{Conjecture}

\newtheorem{lemma}[thm]{Lemma}

\newtheorem{proposition}[thm]{Proposition}

\title{On the Constants of the Lang-Trotter Conjecture for CM Elliptic Curves}
\date{}
\author{Anish Ray}
\begin{document}
	\maketitle	\textbf{Abstract.}
	 In $2021$, Daqing Wan and Ping Xi \cite{wan2021langtrotter} studied the equivalence of the Lang-Trotter conjecture for CM elliptic curves and the Hardy-Littlewood conjecture for primes represented by a quadratic polynomial \cite{Hardy1923SomePO}. Wan and Xi provided an alternative description of the Lang-Trotter Conjecture under the Hardy-Littlewood Conjecture. They obtained an explicit constant $\overline{\omega}_{E,r}$ in the asymptotics of the Lang-Trotter Conjecture. They further conjectured that this particular constant would be equal to the constant $C_{E,r}$ in the asymptotics of the original Lang-Trotter Conjecture \cite{lang2006frobenius}. In this paper, we verify the same for $20$ CM elliptic curves, which also establishes the equivalence of the Lang-Trotter Conjecture and the Hardy-Littlewood Conjecture with respect to $r$, for these CM elliptic curves.
	\vspace{0.2in}
	\begin{center}
	\section{\textbf{Introduction}}\label{sec:1}
	\end{center}
	Let $E/\mathbb{Q}$ be an elliptic curve with complex multiplication by an order $\mathcal{O}$ of an imaginary quadratic number field $K$ described by the Weierstrass equation
	\begin{equation}\label{eq:(1.1)}
		y^2=x^3+ax+b,\quad a,b\in\mathbb{Z}.
	\end{equation}
	Denote by $f_E\subset\mathcal{O}_K$ the conductor ideal of $E$, $|f_E|\in\mathbb{N}$ the conductor of $E$, and $\triangle_E=-16(4a^3+27b^2)$ the discriminant of $E$. For any prime $p\nmid |f_E|$ of good reduction, let $E_p$ denote the reduction of $E$ at $\mathbb{F}_p$, and let
	\begin{equation}
		a_p(E)=p+1-|E_p(\mathbb{F}_p)|
	\end{equation}
	denote the trace of the Frobenius endomorphism at $p$.\\
	\\
	In $1976$, Serge Lang and Hale Trotter \cite{lang2006frobenius}, formulated a precise conjecture on the number of primes $p$ of good reduction such that $a_p(E)$ takes a fixed integer value (say) $r\in\mathbb{Z}$.\\
	\begin{conjecture}[The Lang-Trotter Conjecture for CM Elliptic curves,\cite{lang2006frobenius}]\label{con:1.1}
		Let us define $$\pi_{E,r}(x):=|\{p\leq x: a_p(E)=r, p\nmid |f_E|\}|.$$ We assume $r\neq 0$ for CM elliptic curves $E/\mathbb{Q}$. Then,
		\begin{equation}
			\pi_{E,r}(x)\sim C_{E,r}\frac{\sqrt{x}}{\log x},
		\end{equation}
		where the Lang-Trotter constant $C_{E,r}\ge 0$ is described in terms of the Galois representations associated with $E$. When $C_{E,r}=0$, it means that there are only finitely many primes for which $a_p(E)=r$.
	\end{conjecture}
As interpreted by [\cite{Jones2009-yf} Section 2.2, Page 693 (described in detail in section \ref{sec:2.4})],
\begin{equation*}
	C_{E,r}=\frac{m_E}{2\pi}\frac{|(\operatorname{Gal}(K(E[m_E])/K))_r|}{|\operatorname{Gal}(K(E[m_E])/K)|}\prod_{\substack{\ell\nmid m_E\\ \ell\mid r}}\frac{\ell}{\ell-\chi\mathcal{O}(\ell)}\prod_{\substack{\ell\nmid m_E\\ \ell\nmid r}}\frac{\ell^2-(1+\chi\mathcal{O}(\ell))\ell}{(\ell-1)(\ell-\chi\mathcal{O}(\ell))}.
\end{equation*}
\begin{conjecture}[The Hardy-Littlewood Conjecture for Primes represented by Quadratic Polynomials (\cite{Hardy1923SomePO}Conjecture F, Page no. $48$)]\label{con:1.2}
	Let $a,b,c\in\mathbb{Z}$ such that $a>0$, $b^2-4ac$ is not a square and $\gcd(a,b,c)=1$ and also assume that $a+b$ and $c$ are not simultaneously even. Then there exist infinitely many primes of the form $am^2+bm+c,\, m\in\mathbb{N}$. Moreover, for primes $p$, if we denote
	\begin{equation*}
		\pi_{a,b,c}(x):=|\{p\leq x:p=am^2+bm+c, m\in\mathbb{N},\text{ with }a,b,c\text{ as given above}\}|,
	\end{equation*}
	then
	\begin{equation}
		\pi_{a,b,c}(x)\sim\frac{\gcd(2,a+b)\delta}{\sqrt{a}\phi(\delta)}\prod_{p\nmid2a}\left(1-\frac{\left(\frac{b^2-4ac}{p}\right)}{p-1}\right)\frac{\sqrt{x}}{\log x},
	\end{equation}
	where $\delta$ denotes the odd part of $\gcd(a,b)$, $\phi$ is the Euler's totient function, and $\left(\frac{.}{p}\right)$ denotes the Legendre symbol.
\end{conjecture}
\begin{lemma}[Deuring, \cite{Deuring1941}]\label{lem:1.3}
	Let $E/\mathbb{Q}$ be an elliptic curve with complex mltiplication by $\mathcal{O}_K$ the ring of integers of an imaginary quadratic field $K=\mathbb{Q}(\sqrt{-D})$. Let $p$ be a prime where $E$ has good reduction. If $p$ is inert in $\mathcal{O}_K$, then $a_p(E)=0$. If $p=\pi\overline{\pi}$ splits in $\mathcal{O}_K$, then $a_p(E)=\pi+\overline{\pi}$, where $\pi\in\mathcal{O}_K$ so that the endomorphism $[\pi]$ induces the $p$-th power Frobenius automorphism of the reduction $E$ modulo $p$.
\end{lemma}
Using this lemma we are ready to relate Conjecture \ref{con:1.1} to primes represented by quadratic polynomials. We consider two cases:\\
\textbf{Case 1.} Let $D\equiv1,2\pmod4$. If $p=\pi\overline{\pi}$ splits in $\mathcal{O}_K$ with $\pi=m+n\sqrt{-D}\in\mathcal{O}_K$, $m,n\in\mathbb{Z}$ then one obtains $p=(r/2)^2+Dn^2$ for $a_p(E)=2m=r$.\\
\textbf{Case 2.} Let $D\equiv 3\pmod4$. If $p=\pi\overline{\pi}$ splits in $\mathcal{O}_K$ with $\pi=m+n\frac{1+\sqrt{-D}}{2}\in\mathcal{O}_K$, $m,n\in\mathbb{Z}$, then one obtains $p=(D+1)(r/2)^2-Drm+Dm^2$, where $a_p(E)=2m+n=r$.\\
Using Lemma \ref{lem:1.3} Wan and Xi described the constant in the asymptotics of the Lang-Trotter Conjecture without interpreting the images of the Galois representations. Rather they used the orthogonality of characters and standard methods from analytic number theory. For details, one may refer to Chapters $4$, $5$, and $6$ of \cite{wan2021langtrotter}.
\begin{theorem}[Wan-Xi, \cite{wan2021langtrotter}, Theorem $1.4$]\label{thm:1.4}
	Let $r$ be a fixed non-zero integer and $E/\mathbb{Q}$ be an elliptic curve with CM by the imaginary quadratic field $\mathbb{Q}(\sqrt{-D})$ with squarefree $D\ge1$. The Hardy-Littlewood conjecture implies
	\begin{equation}
		\pi_{E,r}(x)\sim\overline{\omega}_{E,r}\frac{\sqrt{x}}{\log x},
	\end{equation}
	where $\overline{\omega}_{E,r}\ge0$ is some explicit constant defined using analytic number theory.
\end{theorem}
And, they proposed the following conjecture:
\begin{conjecture}[Wan-Xi, \cite{wan2021langtrotter}, Conjecture $1.5$]\label{con:1.5}
	Let $E/\mathbb{Q}$ be a CM elliptic curve and $r\neq 0$ be a fixed integer. Then $\overline{\omega}_{E,r}=C_{E,r}$.
\end{conjecture}
Our goal is to verify the same for the following seven isogeny classes of CM elliptic curves along with two CM elliptic curves having CM from $\mathbb{Q}(\sqrt{-7})$, given by:
\begin{equation*}
	\text{Class }1:\begin{cases}
		E_1: y^2=x^3+4x, & \mathcal{O}=\mathbb{Z}[\sqrt{-1}]\\
		y^2=x^3-x, & \mathcal{O}=\mathbb{Z}[\sqrt{-1}]\\
		y^2=x^3-11x-14, & \mathcal{O}=\mathbb{Z}[\sqrt{-4}]\\
		E_1^*: y^2=x^3-11x+14, & \mathcal{O}=\mathbb{Z}[\sqrt{-4}],
	\end{cases},
\end{equation*}
\begin{equation*}
	\text{Class }2:\begin{cases}
		E_2: y^2=x^3+1, & \mathcal{O}=\mathbb{Z}\biggl[\frac{1+\sqrt{-3}}{2}\biggr]\\
		y^2=x^3-27, & \mathcal{O}=\mathbb{Z}\biggl[\frac{1+\sqrt{-3}}{2}\biggr]\\
		E_2^*: y^2=x^3-15x+22, & \mathcal{O}=\mathbb{Z}[\sqrt{-3}]\\
		y^2=x^3-135x-594 & \mathcal{O}=\mathbb{Z}[\sqrt{-3}]	
	\end{cases},
\end{equation*}
\begin{align*}
	E_3: y^2+xy=x^3-x^2-2x-1\\
	y^2+xy=x^3-x^2-107x+552
\end{align*}
\begin{equation*}
	\text{Class }3:\begin{cases}
		E_4: y^2+y=x^3-x^2-7x+10\\
		y^2+y=x^3-x^2-887-10143
	\end{cases},
\end{equation*}
\begin{equation*}
	\text{Class }4:\begin{cases}
		E_5: y^2+y=x^3-38x+90\\
		y^2+y=x^3-13718x-619025
	\end{cases},
\end{equation*}
\begin{equation*}
	\text{Class }5:\begin{cases}
		E_6: y^2+y=x^3-860x+9707\\
		y^2+y=x^3-1590140x-771794326
	\end{cases},
\end{equation*}
\begin{equation*}
	\text{Class }6:\begin{cases}
		E_7: y^2+y=x^3-7370x+243528\\
		y^2+y=x^3-3308930x-73244287055
	\end{cases},
\end{equation*}
and
\begin{equation*}
	\text{Class }7:\begin{cases}
		E_8: y^2+y=x^3-2174420x+1234136692\\
		y^2+y=x^3-57772164980x-5344733777551611
	\end{cases}
\end{equation*}
having complex multiplication from $\mathbb{Q}(\sqrt{-1})$, $\mathbb{Q}(\sqrt{-3})$, $\mathbb{Q}(\sqrt{-7})$, $\mathbb{Q}(\sqrt{-11})$, $\mathbb{Q}(\sqrt{-19})$, $\mathbb{Q}(\sqrt{-43})$, $\mathbb{Q}(\sqrt{-67})$, and $\mathbb{Q}(\sqrt{-163})$ respectively.
Our methodology is as follows:\\
\textbf{1}. We choose two curves from classes $1$ and $2$, namely, the curve $E_1: y^2=x^3+4x$, the curve $E_1^*:y^2=x^3-11x+14$, and the curve $E_2: y^2=x^3+1$, $E_2^*=y^2=x^3-15x+22$, respectively.\\
\textbf{2}. We compute the constant $\overline{\omega}_{E,r}$ as described in section \ref{sec:2.1} for $E_1$ and section \ref{sec:2.2} for $E_2$, respectively. Then, we compute the constant $C_{E,r}$ for $E_1^*$ and $E_2^*$. We will finally show that these two are equal, and using the fact that the number of $\mathbb{F}_q$-rational points (i.e., $\#E(\mathbb{F}_q)$) of two isogenous elliptic curves are equal, we will get the desired result.\\
\textbf{3}. Then we transform the remaining curves in the form $y^2=4x^3+ax+b$, where $a$ and $b$ are integers as given in section \ref{sec:2.3}. And, then we apply the same techniques as in the case of the curves in classes $1$ and $2$, to obtain the equality of the constants.\\  
\\
But to complete step $2$, we need to compute the constant $m_E$ for the respective curves. To find $m_E$, one has to know the Galois groups as mentioned in the formula of Jones\cite{Jones2009-yf}. For this, we will use the results of \cite{lozano2022galois} which are mentioned in section \ref{sec:3}, where one can obtain the image of the $p$-adic Galois representation $$\rho_{E,p^{\infty}}:\operatorname{Gal}(\overline{\mathbb{Q}}/\mathbb{Q})\to\operatorname{GL}_2(\mathbb{Z}_p)$$ for any prime $p\ge2$ up to conjugation.\\
\\
Eventually, in section \ref{sec:5} we will obtain $m_{E_1^*}=4$ and $m_{E_2^*}=12$, which leads us to the Galois group $\operatorname{Gal}(K(E_2^*[12])/K)$. Then from Galois Theory, we know that this group could be shown isomorphic to $\operatorname{Gal}(K(E_2^*[4])/K)\times\operatorname{Gal}(K(E_2^*[3])/K)$ if $K(E_2^*[4])\cap K(E_2^*[3])=K$. This is assured by Campagna and Pengo\cite{campagna2022entanglement}, where they have given a decomposition of these Galois groups in terms of direct products by studying the images of the natural inclusion
\begin{equation*}
	\operatorname{Gal}(F(E)_{tors}/F)\hookrightarrow\prod_p\operatorname{Gal}(F(E[p^{\infty}])/F),
\end{equation*} 
where $E/F$ is an elliptic curve with CM defined over a number field $F$. They further showed that such exact decomposition is only possible for $30$ isomorphic CM elliptic curves which are mentioned in Table $1$ \cite{campagna2022entanglement}.\\
\\
Using all such results, we would be able to extend the work of H. Qin \cite{qin2021}, who proved that the Lang-Trotter Conjecture \ref{con:1.1} and the Hardy-Littlewood Conjecture \ref{con:1.2} are equivalent for elliptic curves of the form $y^2=x^3+Dx$, $D\in\mathbb{Z}$. Moreover, we would also like to mention that Conjecture $1.5$ has not been verified before for any CM elliptic curve as described in \ref{eq:(1.1)}. We would also like to state that, although the application of Lemma \ref{lem:1.3} is sufficient to form the relation between Conjectures \ref{con:1.1} and \ref{con:1.2} for elliptic curves with CM as described in \ref{eq:(1.1)}, it is not sufficient for general CM elliptic curves since one have to compute such relations for several quadratic polynomials, which could be done but the process is cumbersome.\\
\\
\textbf{\emph{Acknowledgments:}}I would like to thank C. Deninger for suggesting this exciting problem, and for his valuable suggestions during the preparation of this manuscript.
 	\begin{center}
 	\section{\textbf{The constants $\overline{\omega}_{E,r}$ and $C_{E,r}$}}\label{sec:2}
 	\end{center}
 	We consider the curves $E/\mathbb{Q}$ which has CM by the maximal order and one can observe that the ring of integers $\mathcal{O}_K$ where $K=\mathbb{Q}(\sqrt{-D})$ can be one of the $9$ imaginary quadratic fields with class number $1$, namely:
 	\begin{equation*}
 		D\in\{1,2,3,7,11,19,43,67,163\},
 	\end{equation*}
 	since the $j$-invariant is a rational number if and only if $K$ has class number $1$. The non-maximal orders are of the form $\mathcal{O}=\mathbb{Z}+f\mathcal{O}_K$ where $f$ is the conductor of the order $\mathcal{O}$ and $f$ is always $1$ for $\mathcal{O}_K$. For a CM elliptic curve $E/\mathbb{Q}$, the conductor $f$ for the endomorphism ring or the order $\mathcal{O}$ is $1$ for all the $D$'s mentioned above but it can be also $2$ for $D=1,3,7$, and $3$ for $D=3$ which gives us total $13$ different $\overline{\mathbb{Q}}$-isomorphism classes of CM elliptic curves defined over $\mathbb{Q}$ (as given in \cite{silverman1994advanced}, page $483$), and each class has infinitely many family of twists of one elliptic curve over $\mathbb{Q}$. Considering [\cite{silverman1994advanced}, exercise $2.12$], it suffices to study only $9$ of these $13$ classes of CM elliptic curves having the full ring of integer of $K$ as the endomorphism ring.\\
 	\\
 	\subsection{The Constant $\overline{\omega}_{E,r}$ for curves having CM from $\mathbb{Q}(\sqrt{-1})$}\label{sec:2.1}
 	In this section, we will mention an explicit description of the constant $\overline{\omega}_{E,r}$ where $E/\mathbb{Q}$ has CM from $\mathbb{Q}(\sqrt{-1})$. Let $E/\mathbb{Q}$ be defined by
 	\begin{equation}\label{eq:(2.1)}
 		y^2=x^3-gx
 	\end{equation}
 	where $g=(-1)^{\delta}2^{\lambda}g_1$ with
 	\begin{itemize}
 		\item$\delta\in\{0,1\}$,
 		\item$\lambda\in\mathbb{N}\cup\{0\}$,
 		\item$2\nmid g_1\in\mathbb{N}$.
 	\end{itemize}
 	In the next few sections we will define a constant $\Omega_{j}$ which has been described in detail in Chapter $7$ of \cite{wan2021langtrotter}. For this section we define $\Omega_{j}(1;g_1;r):=\prod\limits_{\substack{p^{\nu}||g_1,p\nmid r\\j\nmid\nu}}\frac{-1}{p-1-\left(\frac{-1}{p}\right)}$, $j\in\{2,4\}$ and $\left(\frac{.}{p}\right)$ denotes the Legendre symbol for an odd prime $p$. We denote $\mathfrak{U}_1=\{\pm1,\pm i\}$, the group of roots of unity of $\mathbb{Z}[\sqrt{-1}]$. We would also like to mention that for a prime $p$ dividing an integer $n$, $p||n$ means $p^2\nmid n$, and $\operatorname{Re}(z)$ denotes the real part of $z\in\mathbb{C}$. Wan and Xi, mentioned a constant $h_{1,r}$ in the [\cite{wan2021langtrotter}, Theorem $4.1$], the definition of which could be found in [\cite{wan2021langtrotter}, chapter $2$, page $22(2.15)$] and [\cite{wan2021langtrotter}, page $19 (2.11)$] based on which we restate the theorem as follows:
 	\begin{theorem}[Wan-Xi, \cite{wan2021langtrotter}, Theorem $4.1$]\label{thm:2.1}
 		Suppose $E/\mathbb{Q}$ as defined in \ref{eq:(2.1)} is a CM elliptic curve. Then, for each integer $r\neq 0$, the Hardy-Littlewood conjecture implies that
 		\begin{equation}
 			\pi_{E,r}(x)\sim \overline{\omega}_{E,r}\frac{\sqrt{x}}{\log x},
 		\end{equation}
 		holds with $\overline{\omega}_{E,r}=\frac{1}{4}\kappa(g,r)\prod\limits_{p\nmid2r}\left(1-\frac{\left(\frac{-1}{p}\right)}{p-1}\right)$, where $2\mid r$,
 		and
 		\begin{equation*}
 			\kappa(g,r)=\begin{cases}
 				1-(-1)^{\frac{\lambda r}{4}}\Omega_2+(-1)^{\frac{r}{4}\left(\delta+\frac{g_1-1}{2}\right)}\left(1-(-1)^{\frac{r}{4}}\right)\operatorname{Re}\left(i^{1+\frac{\lambda r}{4}}\right)\Omega_4, & \text{if } 4\mid r\\
 				1+\Omega_2+(-1)^{\frac{r-2}{4}+\frac{g_1^2-1}{8}}\left(1-(-1)^{\delta+\frac{\lambda+g-1}{2}}\right)\Omega_4 & \text{if }2||r,2\mid\lambda\\
 				1 & \text{if } 2||r,2\nmid\lambda.
 			\end{cases}
 		\end{equation*}
 		Unconditionally, $\pi_{E,r}$ is bounded by $O(1)$ if $2\nmid r$.
 	\end{theorem}
 	\subsection{The Constant $\overline{\omega}_{E,r}$ for curves having CM from $\mathbb{Q}(\sqrt{-3})$}\label{sec:2.2}
 	In this section, we will mention an explicit description of the constant $\overline{\omega}_{E,r}$ where $E/\mathbb{Q}$ has CM from $\mathbb{Q}(\sqrt{-3})$. Let $E/\mathbb{Q}$ be defined by
 	\begin{equation}\label{eq:(2.3)}
 		y^2=x^3+g
 	\end{equation}
 	where $g=(-1)^{\delta}2^{\lambda}3^{\mu}g_1$ with
 	\begin{itemize}
 		\item$\delta\in\{0,1\}$,
 		\item$\lambda,\mu\in\mathbb{N}\cup\{0\}$,
 		\item$g_1\in\mathbb{N}$ and $\gcd(6,g_1)=1$.
 	\end{itemize}
 	Moreover, for this section we define $\Omega_{j}(3;g_1;r):=\prod\limits_{\substack{p^{\nu}||g_1,p\nmid r\\j\nmid\nu}}\frac{-1}{p-1-\left(\frac{-3}{p}\right)}$, $j\in\{2,3,6\}$. Let $\mathfrak{U}_3=\{\pm1,\pm\overline{\omega},\pm\overline{\omega}^2\}$, the group of roots of unity of $\mathbb{Z}[\frac{-1+\sqrt{-3}}{2}]$, $\overline{\omega}=e^{\frac{2\pi i}{3}}$.
 	\begin{theorem}[Wan-Xi, \cite{wan2021langtrotter}, Theorem $5.1$]\label{thm:2.2}
 		Suppose $E/\mathbb{Q}$ as defined in \ref{eq:(2.3)} is a CM elliptic curve. Then, for each integer $r\neq 0$, the Hardy-Littlewood conjecture implies that
 		\begin{equation}
 			\pi_{E,r}(x)\sim \overline{\omega}_{E,r}\frac{\sqrt{x}}{\log x},
 		\end{equation}
 		holds with $\overline{\omega}_{E,r}=\frac{\sqrt{3}}{12}\xi(3,r)\prod\limits_{p\nmid2r}\left(1-\frac{\left(\frac{-3}{p}\right)}{p-1}\right)\kappa(g,r)$, where $3\nmid r$, $\xi(3,r)=\begin{cases}
 			1 & \text{if }2\mid r\\
 			2 & \text{if } 2\nmid r,
 		\end{cases}$ $$\kappa(g,r)=\left(1+\zeta_2(g,r)\left(\frac{3}{g_1}\right)\Omega_2+\hat{\mu}\zeta_1(g,r)\left(\Omega_3+\zeta_2(g,r)\left(\frac{3}{g_1}\right)\Omega_6\right)\right).$$
 		In $\kappa(g,r)$, $\hat{\mu}=\frac{2}{3}$ if $3\mid\mu$, otherwise it vanishes, and
 		\begin{equation*}
 			\zeta_1(g,r)=\begin{cases}
 				1+2\operatorname{Re}(\overline{\omega}^{\mu})\operatorname{Re}\left(\overline{\omega}^{\lambda+\frac{g_1^2-1}{3}}\right) & \text{if }r\equiv\pm2\pmod6\\
 				\operatorname{Re}\left(\overline{\omega}^{1+\frac{g_1^2-1}{3}}\right)+\operatorname{Re}(\overline{\omega}^{\mu})\operatorname{Re}\left(\overline{\omega}^{2-\lambda+\frac{g_1^2-1}{3}}+\overline{\omega}^{2+\lambda}\right) & \text{if } r\equiv\pm5\pmod6,
 			\end{cases}
 		\end{equation*} 
 		\begin{equation*}
 			\zeta_2(g,r)=\begin{cases}
 				\pm(-1)^{\delta+\lambda+\mu+\frac{g_1-1}{2}} & \text{if }r\equiv\pm8\pmod{24}\\
 				\pm(-1)^{\delta+\mu+\frac{g_1-1}{2}} & \text{if }r\equiv\pm20\pmod{24}\\
 				\pm\frac{1+(-1)^{\lambda}}{2} & \text{if } r\equiv\pm2\pmod{12}\\
 				\pm\frac{\left(1+(-1)^{\lambda}\right)\left(1+(-1)^{\delta+\mu+\frac{g_1-1}{2}}\right)}{4} & \text{if } r\equiv\pm5\pmod{6},
 			\end{cases}
 		\end{equation*} and $\left(\frac{.}{g_1}\right)$ denotes the Jacobi symbol. In $\omega_{E,r}$, $\left(\frac{.}{p}\right)$ denotes the Legendre symbol for an odd prime $p$. Unconditionally, $\pi_{E,r}$ is bounded by $O(1)$ if $3\mid r$.
 	\end{theorem}
 	\subsection{The Constant $\overline{\omega}_{E,r}$ for curves having CM from other imaginary  quadratic fields}\label{sec:2.3}
 	In this section, we will mention an explicit description of the constant $\overline{\omega}_{E,r}$ where $E/\mathbb{Q}$ has CM from $\mathbb{Q}(\sqrt{-D})$, $D\in\{2,7,11,19,43,67,163\}$. Let $E/\mathbb{Q}$ be defined by
 	\begin{equation}\label{eq:(2.5)}
 		y^2=4x^3+ax+b
 	\end{equation}
 	where $a,b$ will be given in terms of arbitrary integer $g$ as below such that the curve is non-singular:
 	\begin{itemize}
 		\item $D=2:\, (a,b)=\left(-\frac{40}{3}g^2,-\frac{224}{27}g^3\right)$
 		\item $D=7:\, (a,b)=\left(-\frac{35}{4}g^2,-\frac{49}{8}g^3\right)$
 		\item $D=11:\, (a,b)=\left(-\frac{88}{3}g^2,\frac{847}{27}g^3\right)$
 		\item $D=19;\, (a,b)=\left(-152g^2,361g^3\right)$
 		\item $D=43:\, (a,b)=\left(-3440g^2,38829g^3\right)$
 		\item $D=67:\, (a,b)=\left(-29480g^2,974113g^3\right)$
 		\item $D=163:\, (a,b)=\left(8697680g^2,185801\cdot163^2g^3\right)$.
 	\end{itemize}
 	The choices of $(a,b)$ above, make the expression of $a_p(E)$ take a uniform shape when $D\ge7$. And, $g$ is defined as $g=(-1)^{\delta}2^{\lambda}D^{\mu}g_1$, where
 	\begin{itemize}
 		\item $\delta\in\{0,1\}$,
 		\item $\lambda,\mu\in\mathbb{N}\cup\{0\}$,
 		\item $g_1\in\mathbb{N}$ with $\gcd(2D,g_1)=1$,
 	\end{itemize}
 	where we always assume $\mu= 0$ for $D = 2$, and let $\Omega_{j}(D;g_1;r):=\prod\limits_{\substack{p^{\nu}||g_1,p\nmid r\\j\nmid\nu}}\frac{-1}{p-1-\left(\frac{-D}{p}\right)}$.
 	\begin{theorem}[Wan-Xi, \cite{wan2021langtrotter}, Theorem $6.1$]
 		Let $D = 2$. Suppose $E/\mathbb{Q}$ is the elliptic curve defined by \ref{eq:(2.5)} with the above convention. For each non-zero integer $r\equiv2\pmod4$, the Hardy–Littlewood conjecture implies that
 		\begin{equation}
 			\pi_{E,r}(x)\sim \overline{\omega}_{E,r}\frac{\sqrt{x}}{\log x},
 		\end{equation}
 		holds with $$\overline{\omega}_{E,r}=\frac{1}{\sqrt{2}}\left(1+\frac{1}{2}(-1)^{\frac{(r-2)(r+10)}{32}+\delta+\lambda+\frac{g_1-1}{2}}\left(\frac{2}{g_1}\right)\Omega_2(2;g_1,r)\right)\prod\limits_{p\nmid2r}\left(1-\frac{\left(\frac{-2}{p}\right)}{p-1}\right),$$
 		where $\left(\frac{.}{g_1}\right)$ denotes the Jacobi symbol. Unconditionally, $\pi_{E,r}$ is bounded by $O(1)$ if $r\not\equiv2\pmod4$.
 	\end{theorem}
 	\begin{theorem}[Wan-Xi, \cite{wan2021langtrotter}, Theorem $6.2$]
 		Let $D\in\{7, 11, 19, 43, 67, 163\}$. Suppose $E/\mathbb{Q}$ is the elliptic curve defined by \ref{eq:(2.5)} with the above convention. For each non-zero integer $r$, the Hardy–Littlewood conjecture implies that
 		\begin{equation}
 			\pi_{E,r}(x)\sim \overline{\omega}_{E,r}\frac{\sqrt{x}}{\log x},
 		\end{equation}
 		holds with $$\overline{\omega}_{E,r}=\frac{\xi(D,r)\sqrt{D}}{2\phi(D)}\left(1+\xi_D(g,r)\left(\frac{2^{\lambda+1}g_1r}{D}\right)\Omega_2(D;g_1.r)\right)\prod\limits_{p\nmid2r}\left(1-\frac{\left(\frac{-D}{p}\right)}{p-1}\right),$$
 		for $D\nmid r$, where $$\xi(D,r)=\begin{cases}
 			1, & \text{if }D\equiv1\pmod4, 2\mid r\text{ and }\gcd(D,r)=1\\
 			1, & \text{if }D\equiv2\pmod4, 2\mid r\text{ and }\gcd(D,r/2)=1\\
 			1, & \text{if }D\equiv3\pmod4, 2\mid r\text{ and }\gcd(D,r)=1\\
 			2, & \text{if }D\equiv3\pmod8, 2\nmid r\text{ and }\gcd(D,r)=1\\
 			0, & \text{otherwise}
 		\end{cases}$$
 		$$\xi_D(g,r)=\begin{cases}
 			\frac{1}{2}(-1)^{\frac{g_1-1}{2}}(1+(-1)^{\lambda}), & \text{for }2\mid\mid r,\\
 			(-1)^{\delta+\mu+\frac{\lambda r}{4}}, & \text{for }4\mid r,\\
 			\frac{1}{4}(1+(-1)^{\lambda})((-1)^{\frac{g_1-1}{2}}+(-1)^{\delta+\mu}) & \text{for }2\nmid r,
 		\end{cases}$$
 		$\left(\frac{.}{D}\right)$ denotes the Jacobi symbol, and $\phi(.)$ denotes the Euler's totient function. Unconditionally, $\pi_{E,r}$ is bounded by $O(1)$ if $D\mid r$.
 	\end{theorem}
 	\subsection{The Constant $C_{E,r}$}\label{sec:2.4}
 	Let $E[n]$ be the $n$-torsion subgroup. The action of the absolute Galois group on $E[n]$ induces a continuous group homomorphism
 	\begin{equation*}
 		\rho_{E,n}:\operatorname{Gal}(\overline{\mathbb{Q}}/\mathbb{Q})\to\operatorname{Aut}(E[n])\cong\operatorname{GL}_2(\mathbb{Z}/n\mathbb{Z}).
 	\end{equation*}
 	The image of $\rho_{E,n}$ is a subgroup of $\operatorname{GL}_{2}(\mathbb{Z}/n\mathbb{Z})$ since the $n$- torsion group is a free $\mathbb{Z}/n\mathbb{Z}$-module of rank $2$. Taking the inverse limit over all $n\in\mathbb{N}$, one obtains a continuous group homomorphism
 	\begin{equation*}
 		\rho_{E}:\operatorname{Gal}(\overline{\mathbb{Q}}/\mathbb{Q})\to\operatorname{GL}_{2}({\widehat{\mathbb{Z}}})
 	\end{equation*}
 	where $\widehat{\mathbb{Z}}=\lim\limits_{\overleftarrow{n}}\mathbb{Z}/n\mathbb{Z}$ denotes the profinite completion of $\mathbb{Z}$.
 	\begin{theorem}[Serre's open image theorem, \cite{serre1972proprietes}]\label{thm:2.5}
 		Let $E/\mathbb{Q}$ be an elliptic curve with no complex multiplication. Then, we have
 		\begin{equation*}
 			[\operatorname{GL}_{2}(\widehat{\mathbb{Z}}):\operatorname{Gal}(\overline{\mathbb{Q}}/\mathbb{Q})]<\infty.
 		\end{equation*}
 	\end{theorem}
 	By Theorem \ref{thm:2.5}, the image of $\rho_{E}$ has finite index in $\operatorname{GL}_{2}(\widehat{\mathbb{Z}})$ if $E$ does not have CM. In other words, there exists a positive integer level $m$ such that, for the natural projection
 	\begin{equation*}
 		\pi:\operatorname{GL}_{2}(\widehat{\mathbb{Z}})\to\operatorname{GL}_{2}(\mathbb{Z}/n\mathbb{Z}),
 	\end{equation*}
 	one obtains,
 	\begin{equation}
 		\rho_{E}(\operatorname{Gal}(\overline{\mathbb{Q}}/\mathbb{Q}))=\pi^{-1}(G_m(E)),
 	\end{equation} 
 	where $G_m(E)=\operatorname{Gal}(\mathbb{Q}(E[m])/\mathbb{Q})$ is the image of the Galois representation $\rho_{E,m}$ as a subgroup of $\operatorname{GL}_{2}(\mathbb{Z}/m\mathbb{Z})$ for a non-CM elliptic curve $E$. For the CM case, the image of the $\rho_{E}$ is not an open subgroup of $\operatorname{GL}_{2}(\widehat{\mathbb{Z}})$, rather it has an an abelian subgroup of index $2$. We will now follow the description of $C_{E,r}$ for CM elliptic curve $E/\mathbb{Q}$ as given by Jones\cite{Jones2009-yf} by following the computations in \cite{lang2006frobenius}. \\
 	\newline
 	Let $E/\mathbb{Q}$ be an elliptic curve with CM by an order $\mathcal{O}$ in an imaginary quadratic field $K$. Then the torsion part of $E$
 	\begin{equation*}
 		E_{\text{tors}}=\bigcup_{n\ge1}E[n]
 	\end{equation*}
 	is a one-dimensional $\widehat{\mathcal{O}}$-module, where $\widehat{\mathcal{O}}=\lim\limits_{\overleftarrow{n}}\mathcal{O}/n\mathcal{O}$. Then $\operatorname{Gal}(\overline{\mathbb{Q}}/K)$ acts on $E_{\text{tors}}$, preserving the $\widehat{\mathcal{O}}$ action. Thus the image of $\rho_{E}$ restricted to $\operatorname{Gal}(\overline{\mathbb{Q}}/K)$ maps into $(\widehat{\mathcal{O}})^{\times}$:
 	\begin{equation*}
 		\rho_{E}:\operatorname{Gal}(\overline{\mathbb{Q}}/K)\to(\widehat{\mathcal{O}})^{\times}=\operatorname{GL}_1(\widehat{\mathcal{O}}).
 	\end{equation*}
 	Now, we would like to mention a theorem which is a classical analogue of Theorem \ref{thm:2.5}.
 	\begin{theorem}[Serre, \cite{serre1972proprietes}, section $4.5$]\label{thm:2.6}
 		Let $E/\mathbb{Q}$ be an elliptic curve with CM by an order $\mathcal{O}$ in an imaginary quadratic field $K$. Then 
 		\begin{equation*}
 			[(\widehat{\mathcal{O}})^{\times}:\rho_E(\operatorname{Gal}(\overline{\mathbb{Q}}/K))]<\infty.
 		\end{equation*}
 	\end{theorem}
 	Theorem \ref{thm:2.6} states that, $\operatorname{Gal}(K(E[n])/K)$ is a subgroup of $(\mathcal{O}/n\mathcal{O})^{\times}$, and as an analogue to Theorem \ref{thm:2.5}, there is a positive integer $m$ such that, for each positive integer $n$, there is a canonical projection
 	\begin{equation*}
 		\pi:(\mathcal{O}/n\mathcal{O})^{\times}\to(\mathcal{O}/\gcd(n,m)\mathcal{O})^{\times}
 	\end{equation*}
 	and we obtain
 	\begin{equation}\label{eq:(2.9)}
 		\operatorname{Gal}(K(E[n])/K)\cong\pi^{-1}\operatorname{Gal}(K(E[\gcd(n,m)])/K).
 	\end{equation}
 	We observe that \ref{eq:(2.9)} holds for any multiple of $m$. For a CM elliptic curve $E/\mathbb{Q}$, let $m_E$ be the smallest positive integer that satisfies \ref{eq:(2.9)} and for which $4\ell\mid m_E$ where the primes $\ell$ ramifies in the ring $\mathcal{O}$. We would like to note that by the ramification of primes $\ell$ in an  order $\mathcal{O}$ of an imaginary quadratic order $K$, we mean the primes $\ell\mid\triangle_K$ where $\triangle_K$ denotes the discriminant of $K$ but not the conductor of $\mathcal{O}$. For odd unramified primes $\ell\nmid m_E$, we have
 	\begin{equation}\label{eq:(2.10)}
 		\operatorname{Gal}(K(E[\ell])/K)\cong(\mathcal{O}/\ell\mathcal{O})^{\times}.
 	\end{equation}
 Jones interpreted $C_{E,r}$ as
 	\begin{equation}\label{eq:(2.11)}
 		C_{E,r}=\frac{m_E}{2\pi}\frac{|(\operatorname{Gal}(K(E[m_E])/K))_r|}{|\operatorname{Gal}(K(E[m_E])/K)|}\prod_{\ell\nmid m_E}\frac{\ell\cdot|((\mathcal{O}/\ell\mathcal{O})^{\times})_r|}{|(\mathcal{O}/\ell\mathcal{O})^{\times}|}.
 	\end{equation}
 	The infinite product over primes $\ell\nmid m_E$ is absolutely convergent. Explicitly computing the infinite product, one obtains
 	\begin{align}\label{eq:(2.12)}
 		C_{E,r}=\frac{m_E}{2\pi}\frac{|(\operatorname{Gal}(K(E[m_E])/K))_r|}{|\operatorname{Gal}(K(E[m_E])/K)|}\prod_{\substack{\ell\nmid m_E\\ \ell\mid r}}\frac{\ell}{\ell-\chi\mathcal{O}(\ell)}\prod_{\substack{\ell\nmid m_E\\ \ell\nmid r}}\frac{\ell^2-(1+\chi\mathcal{O}(\ell))\ell}{(\ell-1)(\ell-\chi\mathcal{O}(\ell))}\\
 		=\frac{m_E}{2\pi}\frac{|(\operatorname{Gal}(K(E[m_E])/K))_r|}{|\operatorname{Gal}(K(E[m_E])/K)|}\prod_{\substack{\ell\nmid m_E\\ \ell\mid r}}\frac{\ell}{\ell-\chi\mathcal{O}(\ell)}\prod_{\substack{\ell\nmid m_E\\ \ell\nmid r}}\left(1-\frac{\chi\mathcal{O}(\ell)}{(l-1)(l-\chi\mathcal{O}(l))}\right),
 	\end{align}
 	where
 	\begin{equation}\label{eq:(2.14)}
 		\chi\mathcal{O}(\ell)=\begin{cases}
 			1 & \text{if } \ell \text{ splits in }\mathcal{O},\\
 			-1 & \text{if } \ell \text{ is inert in }\mathcal{O}.
 		\end{cases}
 	\end{equation}
 	Here we would like to remark that the notion of splitting makes sense because $\mathcal{O}$ has class number $1$. Moreover, the character $\chi\mathcal{O}(\ell)$ in \ref{eq:(2.14)} is the Legendre symbol for rings of the form $\mathbb{Z}[\sqrt{d}]$.
 	\begin{center}
 		\section{\textbf{The $p$-adic images of the associated Galois representations upto conjugation}}\label{sec:3}
 	\end{center}
	In this section, we follow Lozano-Robledo's work to note the description of the p-adic Galois representations that are associated with elliptic curves $E$ with CM by an order $\mathcal{O}$ of the imaginary quadratic field $K$ with conductor $f\ge1$ defined over $\mathbb{Q}(j(E))$, where $j(E)=j_{K,f}$ is an arbitrary Galois conjugate of the $j$-invariant $j(\mathbb{C}/\mathcal{O})$, associated to the order $\mathcal{O}$ when regarded as complex lattice. In other words, one can describe the possible images of
	\begin{equation*}
		\rho_{E,p^{\infty}}:\operatorname{Gal}(\overline{\mathbb{Q}(j_{K,f})}/\mathbb{Q}(j_{K,f}))\to\operatorname{GL}_2(\mathbb{Z}_p)
	\end{equation*}
	for any prime $p\ge2$, and any elliptic curve $E/\mathbb{Q}(j_{K,f})$ with CM and $j$-invariant $j_{K,f}\neq0,1728$ as a subgroup of $\operatorname{GL}_2(\mathbb{Z}_p)$, up to conjugation. We will also mention the relevant results that describe the groups $\operatorname{Gal}(K(j_{K,f},E[N])/K(j_{K,f}))$ and $\operatorname{Gal}(K(j_{K,f},E[N])/\mathbb{Q}(j_{K,f}))$, which in turn would help us to understand the results on the description the image of $\rho_{E,p^{\infty}}$ for odd primes $p$ of bad reduction and the $2$-adic images, respectively. We would like to remind again that we only note the work done by Lozano-Robledo in the case of $j_{K,f}\neq 0,1728$ for our computations in section \ref{sec:5}.\\
	\\
	Let $n\ge1$ be an integer, $R$ a commutative ring, and $\delta,\phi\in R$. Then the Cartan subgroup of matrices defined as
	\begin{equation*}
		\mathscr{C}_{\delta,\phi}(R)=\biggl\{\begin{pmatrix}
			a+b\phi & b\\
			\delta b & a
		\end{pmatrix}:a,b\in R,\, a^2+ab\phi-\delta b^2\in R^{\times}\biggr\}
	\end{equation*}
	is a subgroup of $\operatorname{GL}_2(R)$, for any $\delta,\phi\in R$. Moreover, if $R$ is a local domain of characteristic $0$, then the normalizer of $\mathscr{C}_{\delta,\phi}(R)$ is defined as $\mathscr{N}_{\delta,\phi}(R)=\big<\mathscr{C}_{\delta,\phi}(R),c_{\phi}\big>$ where $c_{\phi}=\begin{pmatrix}
		-1 & 0\\ \phi & 1
	\end{pmatrix}$. We denote $\mathscr{C}_{\delta,\phi}(\mathbb{Z}/N\mathbb{Z})$ and $\mathscr{N}_{\delta,\phi}(\mathbb{Z}/N\mathbb{Z})$ as $\mathscr{C}_{\delta,\phi}(N)$ and $\mathscr{N}_{\delta,\phi}(N)$, respectively. Let $N\ge2$, $H_f=K(j_{K,f})$ and $h$ be a fixed Weber function.\\
\begin{lemma}[Lozano-Robledo, \cite{lozano2022galois}, Lemma $6.2$]
	Let $E$ be an elliptic curve with CM by an order $\mathcal{O}$ of an imaginary quadratic field $K$ with conductor $f\ge1$, and $j(E)=j_{K,f}$. Fix a $\mathbb{Z}/N\mathbb{Z}$-basis of $E[N]$ and let $\rho_{E,N}:\operatorname{Gal}(\overline{\mathbb{Q}(j_{K,f})}/\mathbb{Q}(j_{K,f}))\to\operatorname{Aut}(E[N])\cong\operatorname{GL}_2(\mathbb{Z}/N\mathbb{Z})$ be the representation afforded by the action of Galois on $E[N]$. Let $G_{E,K,N}=\rho_{E,N}(\operatorname{Gal}(\overline{H_f}/H_f))$, where $H_f =K(j_{K, f})$, and let $G_{E,N}$ be the image of $\rho_{E,N}$ . Then:\\
	\\$(1) $ $G_{E,N}$ is contained in the normalizer of $G_{E,K,N}$ in $\operatorname{GL}_2(\mathbb{Z}/N\mathbb{Z})$.\\
	$(2)$ If we pick a $\mathbb{Z}/N\mathbb{Z}$-basis of $E[N]$ so that $G_{E,K,N}$ is contained in $\mathscr{C}_{\delta,\phi}(N)$ with $(\delta,\phi)=(\triangle_K,0)$ or $((\triangle_K-1)f^2/4,f)$, then $G_{E,N}$ is contained in $\mathscr{N}_{\delta,\phi}(N)=\big<\mathscr{C}_{\delta,\phi}(N),c_\phi\big>$.
\end{lemma}
\begin{theorem}[Lozano-Robledo, \cite{lozano2022galois}, Theorem $1.1$, Theorem $6.3(3)$]\label{thm:3.2}
	Let $E$ be an elliptic curve with CM by an order $\mathcal{O}$ of an imaginary quadratic field $K$ with conductor $f\ge1$, and $j(E)=j_{K,f}$. Let $N\ge3$, and let $\rho_{E,N}:\operatorname{Gal}(\overline{\mathbb{Q}(j_{K,f})}/\mathbb{Q}(j_{K,f}))\to\operatorname{Aut}(E[N])\cong\operatorname{GL}_2(\mathbb{Z}/N\mathbb{Z})$ be the Galois representation. We define groups of $\operatorname{GL}_2(\mathbb{Z}/N\mathbb{Z})$ as follows:
	\begin{itemize}
		\item If $\triangle_Kf^2\equiv0\pmod4$ or $N$ is odd, let $\delta=\triangle_Kf^2/4$, and $\phi=0$.
		\item If $\triangle_Kf^2\equiv1\pmod4$, and $N$ is even, let $\delta=(\triangle_K-1)f^2/4$, and $\phi=f$.
	\end{itemize}
	If we define $\mathscr{C}_{\delta,\phi}(N)$ by
	\begin{equation*}
		\mathscr{C}_{\delta,\phi}(N)=\biggl\{\begin{pmatrix}a+b\phi & b\\\delta b & a\end{pmatrix}:a,b\in\mathbb{Z}/N\mathbb{Z}, a^2+ab\phi-\delta b^2\in(\mathbb{Z}/N\mathbb{Z})^{\times}\biggr\}\subset\operatorname{GL}_2(\mathbb{Z}/N\mathbb{Z}),
	\end{equation*}
	with $\mathscr{N}_{\delta,\phi}(N)=\biggl<\mathscr{C}_{\delta,\phi}(N),c_\phi=\begin{pmatrix}-1 & 0\\\phi & 1\end{pmatrix}\biggr>$, then there is $\mathbb{Z}/N\mathbb{Z}$-basis of $E[N]$ such that the image of $\rho_{E,N}$ is contained in $\mathscr{N}_{\delta,\phi}(N)$.\\
	Moreover:\\
	\\$(1)$$\mathscr{C}_{\delta,\phi}(N)\cong(\mathcal{O}/N\mathcal{O})^{\times}$ is a subgroup of index $2$ in $\mathscr{N}_{\delta,\phi}(N)$.\\
	$(2)$ The index of the image of $\rho$ in $\mathscr{N}_{\delta,\phi}(N)$ coincides with the order of the Galois group $\operatorname{Gal}(K(j_{K,f},E[N])/K(j_{K,f}),h(E[N]))$ ), for a fixed Weber function $h$, and it is a divisor of the order of $\frac{\mathcal{O}^{\times}}{\mathcal{O}_N^{\times}}$, where $\mathcal{O}_N^{\times}=\{u\in\mathcal{O}^{\times}:u\equiv 1\mod N\mathcal{O}\}$.\\
	$(3)$ The index of $\operatorname{Gal}(H_f(E[N])/H_f)$ in $\mathscr{C}_{\delta,\phi}(N)$ is the same as the index of the Galois group $\operatorname{Gal}(H_f(E[N])/\mathbb{Q}(j_{K,f}))$ in $\mathscr{N}_{\delta,\phi}(N)$.
\end{theorem}
\begin{theorem}[Lozano-Robledo, \cite{lozano2022galois}, Theorem $1.2(1),(2)$]
	Let $E$ be an elliptic curve with CM by an order $\mathcal{O}$ of an imaginary quadratic field $K$ with conductor $f\ge1$, and $j(E)=j_{K,f}$. Let $N\ge3$, and let $\rho_{E,N}:\operatorname{Gal}(\overline{\mathbb{Q}(j_{K,f})}/\mathbb{Q}(j_{K,f}))\to\operatorname{Aut}(E[N])\cong\operatorname{GL}_2(\mathbb{Z}/N\mathbb{Z})$ be the Galois representation. We define groups of $\operatorname{GL}_2(\mathbb{Z}/N\mathbb{Z})$ as follows:
	\begin{itemize}
		\item If $\triangle_Kf^2\equiv0\pmod4$ or $N$ is odd, let $\delta=\triangle_Kf^2/4$, and $\phi=0$.
		\item If $\triangle_Kf^2\equiv1\pmod4$, and $N$ is even, let $\delta=(\triangle_K-1)f^2/4$, and $\phi=f$.
	\end{itemize}
	Then:\\
	\\
	$(1)$ Let $\rho_{E}:\operatorname{Gal}(\overline{\mathbb{Q}(j_{K,f})}/\mathbb{Q}(j_{K,f}))\to\lim\limits_{\overleftarrow{N}}\operatorname{Aut}(E[N])\cong\operatorname{GL}_2(\hat{Z})$ and let $\mathscr{N}_{\delta,\phi}=\lim\limits_{\overleftarrow{N}}\mathscr{N}_{\delta,\phi}(N)$. Then, there is a compatible system of bases of $E[N]$ such that the image of $\rho(E)$ is contained in $\mathscr{N}_{\delta,\phi}$, and the index of the image of $\rho_{E}$ in $\mathscr{N}_{\delta,\phi}$ is a divisor of the order of $\mathcal{O}^{\times}$. In particular, the index is a divisor of $4$ or $6$. Moreover, for every $K$ and $f\ge1$, and a fixed $N\ge3$, there is an elliptic curve $E/\mathbb{Q}(j_{K,f})$ such that the image of $\rho_{E,N}$ is precisely $\mathscr{N}_{\delta,\phi}(N)$.\\
	$(2)$ Let $p>2$ and $j_{K,f}\neq0$, or $p>3$. Let $G_{E,p^{\infty}}\subset\mathscr{N}_{\delta,\phi}(p^{\infty})$ be the image of $\rho_{E,p^{\infty}}$ with respect to a suitable basis of $T_p(E)$, and let $G_{E,p}\subset\mathscr{N}_{\delta,\phi}(p)$  be the image of $\rho_{E,p}\equiv \rho_{E,p^{\infty}}\mod p$ is the full inverse image of $G_{E,p}$ via the reduction $\mod p$ map $\mathscr{N}_{\delta,\phi}(p^{\infty})\to\mathscr{N}_{\delta,\phi}(p)$.
\end{theorem}
Let $E/\mathbb{Q}(j_{K,f})$ be an elliptic curve with CM by an order $\mathcal{O}$ with conductor $f$ in an imaginary quadratic field $K$, and let $N\ge3$. Then there is a Galois representation,
\begin{equation*}
	\rho_{E,N}:\operatorname{Gal}(\overline{\mathbb{Q}(j_{K,f})}/\mathbb{Q}(j_{K,f}))\to\operatorname{Aut}(E[N]).
\end{equation*}
The $N$-torsion of $E$ is a free $\mathcal{O}/N\mathcal{O}$-module of rank one, and as in [\cite{lozano2022galois}, Theorem $4.1$] one has for $H_f=K(j_{K,f})$, the restricted representation:
\begin{equation*}
	\operatorname{Gal}(H_f(E[N])/H_f)\hookrightarrow\operatorname{Aut}_{\mathcal{O}/N\mathcal{O}}(E[N])\cong\left(\frac{\mathcal{O}}{N\mathcal{O}}\right)^{\times}.
\end{equation*}
On the other hand, if one fixes an embedding of $\mathbb{Q}(j_{K,f})$ into $\mathbb{C}$, then complex conjugation in $\mathbb{C}$ induces a complex conjugation $c\in\operatorname{Gal}(\overline{\mathbb{Q}(j_{K,f})}/\mathbb{Q}(j_{K,f}))$.  By [\cite{lozano2022galois}, Lemma $5.5$] the image $\rho_{E,N}(\operatorname{Gal}(\overline{\mathbb{Q}(j_{K,f})}/\mathbb{Q}(j_{K,f})))$ is contained in $\mathscr{N}_{\delta.\phi}(N)$. The following results tell us where a complex conjugation lands via $\rho_{E,N}$.\\
\begin{lemma}[Lozano-Robledo, \cite{lozano2022galois}, Lemma $6.5$]
	Let $c\in\operatorname{Gal}(\overline{\mathbb{Q}(j_{K,f})}/\mathbb{Q}(j_{K,f}))$ be a complex conjugation. Then, $\rho_{E,N}(c)$ is a matrix in $\operatorname{GL}_2(\mathbb{Z}/N\mathbb{Z})$ with determinant $-1\mod N$ and zero trace.
\end{lemma}
Next, one can use [\cite{zarhin1987endomorphisms}, Theorem $1$] to show that the image of a complex conjugation via $\rho_{E,N}$ cannot be contained in a Cartan subgroup, for $N\ge3$. This leads to the following theorem:
\begin{theorem}[Lozano-Robledo, \cite{lozano2022galois}, Theorem $6.7$]
	Let $E/\mathbb{Q}(j_{K,f})$ be an elliptic curve with CM by an order $\mathcal{O}$ with conductor $f$ in an imaginary quadratic field $K$, let $p$ be a prime, let $n=2$ if $p=2$ and $n=1$ if
	$p>2$, and choose a basis of the $p$-adic Tate module $T_p(E)$ as in Theorem \ref{thm:3.2}. Then, the image of $\rho_{E,p^n}:\operatorname{Gal}(\overline{\mathbb{Q}(j_{K,f})}/\mathbb{Q}(j_{K,f}))\to\operatorname{GL}_2(\mathbb{Z}/p^n\mathbb{Z})$ cannot be contained in $\mathscr{C}_{\delta,\phi}(p^n)$.
\end{theorem}
The above theorem shows that the image of a complex conjugation via $\rho_{E,N}$ lands in $\mathscr{N}_{\delta,\phi}(N)/\mathscr{C}_{\delta,\phi}(N)$ for any $N\ge3$.
\begin{theorem}[Lozano-Robledo, \cite{lozano2022galois}, Theorem $6.10$]\label{thm:3.6}
	Let $E/\mathbb{Q}(j_{K,f})$ be an elliptic curve with CM by an order $\mathcal{O}$ with conductor $f$ in an imaginary quadratic field $K$, let $p$ be an odd prime, let $G_{E,p^{\infty}}$ be the image of $\rho_{E,p^{\infty}}$, and let $G_{E,K,p^{\infty}}=\rho_{E,p^{\infty}}(\operatorname{Gal}(\overline{H_f}/H_f))$. Let $c\in\operatorname{Gal}(\overline{\mathbb{Q}(j_{K,f})}/\mathbb{Q}(j_{K,f}))$ be a complex conjugation, and let $\gamma=\rho_{E,p^{\infty}}(c)$. Then:\\
	\\(a) There is a $\mathbb{Z}_p$-basis of $T_p(E)$ such that $G_{E,K,p^{\infty}}\subset\mathscr{C}_{\delta,\phi}(p^{\infty})$ and $\gamma=c_{\varepsilon}=\begin{pmatrix}\varepsilon & 0\\0 & -\varepsilon
	\end{pmatrix}$ for some $\varepsilon\in\{\pm1\}$.\\
	(b) If $\delta=\triangle_{K}f^2/4\not\equiv0\pmod p$, then for each $\varepsilon\in\{\pm1\}$ there is a basis such that $\rho_{E,p^{\infty}}(c)=c_{\varepsilon}$.\\
	(c) If $\delta\equiv0\pmod p$, then the group $\big<\mathscr{C}_{\delta,\phi}(p^{\infty}),c_1\big>$ cannot be conjugated to $\big<\mathscr{C}_{\delta,\phi}(p^{\infty}),c_1^{-1}\big>$ in such a way that sends $c_1$ to $c_1^{-1}$.
\end{theorem}
From [\cite{lozano2022galois}, Lemma $6.11$], one obtains the image of $\rho_{E,p^{\infty}}$, $G_{E,p^{\infty}}=\big<\zeta c_\phi,G_{E,K,p^{\infty}}\big>$ for all primes $p\ge2$, where $\zeta$ is a root of unity in $\mathscr{C}_{\delta,\phi}(p^{\infty})$ of order $d$ (index of $G_{E,K,p^{\infty}}$ in $\mathscr{C}_{\delta,\phi}(p^{\infty})$) and $c_{\phi}=\begin{pmatrix}
	-1 & 0\\
	\phi & 1
\end{pmatrix}$. Also, $d\mid\#\mathcal{O}$. This result forms the basis of describing the image of $\rho_{E,p}$ as obtained by [\cite{lozano2022galois}] for odd primes dividing the conductor of the discriminant. The last theorem in this section describes the image of $\rho_{E,N}$ for $N=2^n$, $n\ge1$, for $j\neq0,1728$.
\begin{theorem}[Lozano-Robledo, \cite{lozano2022galois}, Theorem $8.1(a)$]\label{thm:3.7}
Let $E/\mathbb{Q}(j_{K,f})$ be an elliptic curve with CM by an order $\mathcal{O}$ with conductor $f$ in an imaginary quadratic field $K$, let $p$ be an odd prime dividing $\triangle_{K}f$ (thus, $j\neq1728$ where
$f\triangle_K = -4$), and let $G_{E,p}$ be the image of $\rho_{E,p}$. Then using Theorem \ref{thm:3.6} it follows that, if $j\neq0,1728$ then either $G_{E,p}=\mathscr{N}_{\delta,\phi}(p)$, or $G_{E,p}$ is generated by $c_{\varepsilon}$ and the group  $G_{E,K,p}=\biggl\{\begin{pmatrix}a & b\\0 & a\end{pmatrix}:a\in((\mathbb{Z}/p\mathbb{Z})^{\times})^2,b\in\mathbb{Z}/p\mathbb{Z}\}\leq\operatorname{GL}_2(\mathbb{Z}/p\mathbb{Z})$ where $\delta\equiv 0\pmod p$.
\end{theorem}
\begin{theorem}[Lozano-Robledo, \cite{lozano2022galois}, Theorem $9.3$]\label{thm:3.8}
Let $E/\mathbb{Q}(j_{K,f})$ be an elliptic curve with CM by an order $\mathcal{O}$ with conductor $f$ in an imaginary quadratic field $K$,  and  $j\neq1728$. Suppose that $\operatorname{Gal}(H_f(E[2^n])/H_f)\subsetneq(\mathcal{O}/2^n\mathcal{O})^{\times}$ for some $n\ge1$. Then, for all $n\ge3$, we have $\operatorname{Gal}(H_f(E[2^n])/H_f)\cong(\mathcal{O}/2^n\mathcal{O})^{\times}/\{\pm1\}$, and there are two possibilities:\\
\\
(1) If $\operatorname{Gal}(H_f(E[4])/H_f)\cong(\mathcal{O}/4\mathcal{O})^{\times}/\{\pm1\}$, then:
\begin{itemize}
	\item[(a)] $\triangle_{K}f^2\equiv0\pmod{16}$. In particular, one has either:
	\begin{itemize}
		\item $\triangle_K\equiv1\pmod4$ and $f\equiv0\pmod4$, or
		\item $\triangle_K\equiv0\pmod4$ and $f\equiv0\pmod2$.
	\end{itemize}
	\item[(b)] The $4$-th roots of unity $\mu_4$ are contained in $H_f$, i.e., $\mathbb{Q}(i)\subset H_f$.
	\item[(c)] For each $n\ge2$, there is a $\mathbb{Z}/2^n\mathbb{Z}$-basis of $E[2^n]$ such that the image of the Galois representation $\rho_{E,2^n}:\operatorname{Gal}(\overline{H_f}/H_f)\to\operatorname{GL}_2(\mathbb{Z}/2^n\mathbb{Z})$ is one of the groups
	$$J_1=\biggl<\begin{pmatrix}5 & 0\\0 & 5
	\end{pmatrix},\begin{pmatrix}
		1 & 1\\ \delta & 1
	\end{pmatrix}\biggr>\quad\text{or}\quad J_2=\biggl<\begin{pmatrix}5 & 0\\0 & 5
	\end{pmatrix},\begin{pmatrix}
		-1 & -1\\ -\delta & -1
	\end{pmatrix}\biggr>\leq\mathscr{C}_{\delta,0}(2^n).$$
\end{itemize}	
(2) If $\operatorname{Gal}(H_f(E[4])/H_f)\cong(\mathcal{O}/4\mathcal{O})^{\times}$, then:
\begin{itemize}
	\item[(a)] $\triangle_{K}\equiv0\pmod8$.
	\item[(b)] For each $n\ge3$, there is a $\mathbb{Z}/2^n\mathbb{Z}$-basis of $E[2^n]$ such that the image of the Galois representation $\rho_{E,2^n}:\operatorname{Gal}(\overline{H_f}/H_f)\to\operatorname{GL}_2(\mathbb{Z}/2^n\mathbb{Z})$ is one of the groups
	$$J_1^{\prime}=\biggl<\begin{pmatrix}3 & 0\\0 & 3
	\end{pmatrix},\begin{pmatrix}
		1 & 1\\ \delta & 1
	\end{pmatrix}\biggr>\quad\text{or}\quad J_2^{\prime}=\biggl<\begin{pmatrix}3 & 0\\0 & 3
	\end{pmatrix},\begin{pmatrix}
		-1 & -1\\ -\delta & -1
	\end{pmatrix}\biggr>\leq\mathscr{C}_{\delta,0}(2^n).$$
\end{itemize}
\end{theorem} 
	\section{\textbf{Possible Decompositions of the Galois Groups of the Division Fields into Direct Products of Galois Groups}}\label{sec:4}
	In the previous section, we studied possible representations of the Galois groups of the $p$-division fields for a prime number $p$ as a subgroup of $\operatorname{GL}_2(\mathbb{Z}/p\mathbb{Z})$. But for our computations in section \ref{sec:5} we need to know such representations for the Galois groups of $N$-division fields where $N$ is a composite number. To this end, we note the work of Campagna and Pengo\cite{campagna2022entanglement}.\\
	\\
	Let $E/F$ be an elliptic curve, and let $F\subset\overline{F}$ be a fixed algebraic closure. Let $E_{tors}=E(\overline{F})_{tors}$ be the group of all torsion points of $E$, and let $F(E)_{tors}$ denote the the compositum of the family of fields $\{F(E[p^{\infty}])\}$. Each extension $F\subset F(E[p^{\infty}])$ is defined as the compsitum of the family $\{F(E[p^n])\}$. The absolute Galois group $\operatorname{Gal}(\overline{F}/F)$ acts on the group $E_{tors}$ , which gives us the following Galois representation
	\begin{equation*}
		\rho_E:\operatorname{Gal}(F(E)_{tors}/F)\to\operatorname{Aut}_{\mathbb{Z}}(E_{tors})\cong\operatorname{GL}_2(\hat{\mathbb{Z}}).
	\end{equation*}
	There is a natural inclusion
	\begin{equation}\label{eq:(4.1)}
		\operatorname{Gal}(F(E)_{tors}/F)\hookrightarrow\prod_p\operatorname{Gal}(F(E[p^{\infty}])/F).
	\end{equation}
	In this section, we note the conditions under which this inclusion would be an isomorphism for CM elliptic curves as investigated by Campagna and Pengo\cite{campagna2022entanglement}. To this end, we note the following theorems:
	\begin{theorem}[Campagna-Pengo, \cite{campagna2022entanglement}, Theorem $1.1$]\label{thm:4.1}
		Let $F$ be a number field and $E/F$ an elliptic curve with complex
		multiplication by an order $\mathcal{O}$ with conductor $f$ in an imaginary quadratic field $K\subset F$. Denote by $B_E:=f\triangle_FN_{F/\mathbb{Q}}(f_E)\in\mathbb{Z}$, where $\triangle_{F}\in\mathbb{Z}$ is the absolute discriminant of the number field $F$ and the absolute norm $N_{F/\mathbb{Q}}(f_E)=|\mathcal{O}_F/f_E|\in\mathbb{N}$ of the conductor ideal $f_E\subset\mathcal{O}_F$ of $E$. Then the map \ref{eq:(4.1)} induces an isomorphism
		\begin{equation}\label{eq:(4.2)}
			\operatorname{Gal}(F(E)_{tors}/F)\xrightarrow{\cong}\operatorname{Gal}(F(E[S^{\infty}])/F)\prod_{p\notin S}\operatorname{Gal}(F(E[p^{\infty}])/F).
		\end{equation}
		for any finite set of primes $S\subset\mathbb{N}$ containing the primes dividing $B_E$ and $F(E[S^{\infty}])$ denotes the compositum of the family of fields $\{F(E[p^{\infty}])\}_{p\in S}$. In this case, one says that the family $\{F(E[S^{\infty}])\}_{p\in S}\cup\{F(E[p^{\infty}])\}_{p\notin S}$ is linearly disjoint over $F$. 
	\end{theorem}
	The isomorphism \ref{eq:(4.2)} does not hold if the set $S$ does not contain all the primes dividing $B_E$. However, we obtain a sufficient condition from the following theorem for which the said isomorphism holds when $S$ is empty.
	\begin{theorem}[Campagna-Pengo, \cite{campagna2022entanglement}, Theorem $1.2$]
		Let $\mathcal{O}$ be an order with conductor $f$ in an imaginary quadratic field $K$, and let $E$ be an elliptic curve with complex multiplication by $\mathcal{O}$ defined over the ring class field $H_\mathcal{O}=K(j(E))$, such that $H_{\mathcal{O}}(E_{tors})\neq K^{ab}$. Then the whole family of $p^{\infty}$-division fields $\{H_{\mathcal{O}}(E[p^{\infty}])\}_p$ runs through the rational primes $p$, is linearly disjoint over $H_{\mathcal{O}}$. Moreover, if the class group $Pic({\mathcal{O}})\neq\{1\}$, there exist infinitely many such elliptic
		curves $E/H_{\mathcal{O}}$, which are nonisomorphic over $H_{\mathcal{O}}$.
	\end{theorem} 
	\begin{theorem}[Campagna-Pengo, \cite{campagna2022entanglement}, Theorem $1.3$]\label{thm:4.3}
		Let $\mathcal{O}$ be an order with conductor $f$ of discriminant $\triangle_{\mathcal{O}}:=f^2\triangle_{K}<-4$ inside an imaginary quadratic field $K$, and suppose that $Pic(\mathcal{O})=\{1\}$. Let $E/\mathbb{Q}$ be an elliptic curve with complex multiplication by $\mathcal{O}$. Then the family of division fields $\{K(E[p^{\infty}])\}_p$, where $p$ runs over the rational primes $p$, is linearly disjoint over $K$ if and only if $E$ is isomorphic over $K$ to one of the thirty elliptic curves appearing in [\cite{campagna2022entanglement}, Table $1$].
	\end{theorem}
	We consider the elliptic curves $E_1^*$ and $E_2^*$ as mentioned in section $1$. Referring to [\cite{campagna2022entanglement}, Table $1$] and their case-by-case analysis of these elliptic curves on page no. $57,59$, and $60$ of [\cite{campagna2022entanglement}, Section $6$], one obtains the following results:\\
	(a) For the elliptic curve $E_1^*$ having CM from the order $\mathcal{O}=\mathbb{Z}[\sqrt{-4}]$ in the imaginary quadratic field $K=\mathbb{Q}(\sqrt{-1})$, it's conductor $|f_{E_1^*}|\in\mathbb{N}$ and in fact, it is a power of $2$ which ramifies in $K$. Then from Theorem \ref{thm:4.1}, one obtains
	\begin{equation}
		\operatorname{Gal}(K(E_1^*)_{tors}/K)\cong\prod_{p}\operatorname{Gal}(K(E_1^*[p^{\infty}])/K).
	\end{equation} 
	where $p$ runs over all rational primes $p\in\mathbb{N}.$ Moreover, from [\cite{campagna2022entanglement}, Proposition $3.3$] it follows that $\operatorname{Gal}(K(E_1^*[p^{m}])/K)\cong(\mathcal{O}/p^m\mathcal{O})^{\times}$, where $p\neq 2$, $m\in\mathbb{N}$. Also, using [\cite{campagna2022entanglement}, Proposition $5.5$] along with [\cite{campagna2022entanglement}, Theorem $6.1$], one obtains $\operatorname{Gal}(K(E_1^*[2^{m}])/K)\cong(\mathcal{O}/2^m\mathcal{O})^{\times}/\{\pm1\}$, where $m\in\mathbb{N}$.\\
	(b) For the elliptic curve $E_2^*$ having CM from the order $\mathcal{O}=\mathbb{Z}[\sqrt{-3}]$ in the imaginary quadratic field $K=\mathbb{Q}(\sqrt{-3})$, we have a different scenario since $|f_{E_2^*}|=2^23^2$ which is not a power of a unique prime. So, one has to study the division fields $K(E_2^*[2^{\infty}])$ and $K(E_2^*[3^{\infty}])$. Then $\rho_{E_2^*,3}$ is not surjective since $K(E_2^*[3])=H_{3,\mathcal{O}}=K(\sqrt[3]{2})$ where $H_{I,\mathcal{O}}$ denotes the ray class field of $K$ modulo a non-zero ideal $I\subset\mathcal{O}$ relative to $\mathcal{O}$ (as defined in [\cite{campagna2022entanglement}, Section $4$, Theorem $4.7$]). Then $\rho_{E_2^*,3^n}$ is not surjective for every $n\in\mathbb{N}$. Moreover, $\rho_{E_2^*,2^n}$ is surjective for every $n\in\mathbb{N}$, since [\cite{campagna2022entanglement}, Theorem $4.6$ and Theorem $4.7$] suggests that
	\begin{equation*}
		|(\mathcal{O}/2^n\mathcal{O})^{\times}|=\frac{[H_{2^n3,\mathcal{O}}:K]}{[H_{3,\mathcal{O}}:K]}=\frac{[H_{2^n3,\mathcal{O}}:K]}{[K(E_2^*[3]):K]}\leq[\frac{[K(E_2^*[2^n3]):K]}{[K(E_2^*[3]):K]}\leq[K(E_2^*[2^n3]):K].
	\end{equation*}
	Then [\cite{campagna2022entanglement}, Lemma $3.2$] shows that every inequality is actually an equality. Then, one obtains $K(E_2^*[2^{n}])$ and $K(E_2^*[3^{m}])=K$ for every $n,m\in\mathbb{N}$. Finally, using Theorem \ref{thm:4.1} and [\cite{campagna2022entanglement}, Proposition $3.3$], we have
	\begin{equation}
		\operatorname{Gal}(K(E_2^*)_{tors}/K)\cong\prod_{p}\operatorname{Gal}(K(E_2^*[p^{\infty}])/K).
	\end{equation} where $p$ runs over all primes $p\in\mathbb{N}$. Moreover, $\operatorname{Gal}(K(E_2^*[p^{m}])/K)\cong(\mathcal{O}/p^m\mathcal{O})^{\times}$, where $p\neq 3$, $m\in\mathbb{N}$, and $\operatorname{Gal}(K(E_2^*[3^{m}])/K)\cong(\mathcal{O}/3^m\mathcal{O})^{\times}/\{\pm1\}$, where $m\in\mathbb{N}$.\\
\\
(c) For all the other isogeny classes of CM elliptic curves $E$ as listed above in section \ref{sec:1}, Campagna and Pengo obtained the isomorphism
\begin{equation}
\operatorname{Gal}(K(E[p^m])/K)\cong\begin{cases}\left(\mathcal{O}/p^m\mathcal{O}\right)^{\times}/{\pm1} & \text{if } p\in S\\
	\left(\mathcal{O}/p^m\mathcal{O}\right)^{\times} & \text{otherwise}
\end{cases}
\end{equation}
for all $m\in\mathbb{N}$. Moreover, one obtains
\begin{equation}
	\operatorname{Gal}(K(E)_{tors}/K)\cong\prod_{p}\operatorname{Gal}(K(E[p^{\infty}])/K).
\end{equation}
	Thus, all of these elliptic curves satisfy Theorem \ref{thm:4.3}.
   \section{\textbf{Final Computations}}\label{sec:5}
	Let us recall the curve $E_1/\mathbb{Q}$
	\begin{equation*} 
		E_1:y^2=x^3+4x,
	\end{equation*}
	and the curve $E_2/\mathbb{Q}$
	\begin{equation*}
		y^2=x^3+1.
	\end{equation*}
	In this section, we want to compute the constant $C_{E,r}$ for the curves $E_1$ and $E_2$ which are isogenous to the curves $E_1^*/\mathbb{Q}$ and $E_2^*/\mathbb{Q}$, respectively. Here,
	\begin{equation*}
		E_1^*: y^2=x^3-11x+14,
	\end{equation*}
	and
	\begin{equation*}
		E_2^*: y^2=x^3-15x+22.
	\end{equation*}
	Thanks to the results mentioned in sections \ref{sec:3} and \ref{sec:4}, we have the necessary tools to compute the explicit images of the corresponding Galois images. We would like to mention that using Theorem \ref{thm:4.1}, one can directly compute $C_{E,r}$ for the curves $E_1$ and $E_2$, since the set $S$ in Theorem \ref{thm:4.1} is a singleton set for these curves. There are curves for which the set $S$ would contain multiple elements, but the structure of such Galois groups is unknown. That's why all of our discussion in this section follows as a consequence of Theorem \ref{thm:4.3}, but this approach is quite restrictive, since not all CM elliptic curves $E/\mathbb{Q}$ satisfy this theorem.
	\subsection{Computing $m_E$}\label{sec:5.1}
	We recall the constant $C_{E,r}$ as described in \ref{eq:(2.11)},
	\begin{equation*}
		C_{E,r}=\frac{m_E}{2\pi}\frac{|(\operatorname{Gal}(K(E[m_E])/K))_r|}{|\operatorname{Gal}(K(E[m_E])/K)|}\prod_{\substack{\ell\nmid m_E\\ \ell\mid r}}\frac{\ell}{\ell-\chi\mathcal{O}(\ell)}\prod_{\substack{\ell\nmid m_E\\ \ell\nmid r}}\frac{\ell^2-(1+\chi\mathcal{O}(\ell))\ell}{(\ell-1)(\ell-\chi\mathcal{O}(\ell))}.
	\end{equation*}
	One notices that to proceed any further we need to know the exact value of $m_E$ for the curves $E_1^*$ and $E_2^*$, which has CM from $K=\mathbb{Q}(\sqrt{-1})$ and $K=\mathbb{Q}(\sqrt{-3})$, respectively. We recall that $m_E$ is the smallest positive integer that satisfies the isomorphism \ref{eq:(2.9)} and $4\ell\mid m_E$ where $\ell$ is a prime ramifying in the order $\mathcal{O}$ of $K$ by which $E$ has CM. Hence, $m_{E_1^*}=4$ and $m_{E_2^*}=12$. Next, we have to verify that these values of $m_{E_1^*}$ and $m_{E_2^*}$ satisfy the isomorphism \ref{eq:(2.9)}.
	\begin{proposition}Let $E_1^*/\mathbb{Q}$ and $E_2^*/\mathbb{Q}$ denote the curves,
		\begin{equation*}
			E_1^*: y^2=x^3-11x+14
		\end{equation*}
		with CM by the order $\mathcal{O}=\mathbb{Z}[\sqrt{-4}]$ of the imaginary quadratic field $K=\mathbb{Q}(\sqrt{-1})$, and
		\begin{equation*}
			E_2^*: y^2=x^3-15x+22.
		\end{equation*}
		with CM by the order $\mathcal{O}=\mathbb{Z}[\sqrt{-3}]$ of the imaginary quadratic field $K=\mathbb{Q}(\sqrt{-3})$, respectively. Then $m_{E_1^*}=4$ and $m_{E_2^*}=12$ satisfies the the respective isomorphisms given by \ref{eq:(2.9)}.
	\end{proposition}
	\begin{proof}
		For the elliptic curve $E_1^*$, let $n=2^{\alpha}\times\beta$ and $\beta$ is odd. Consider the natural group homomorphism
		\begin{equation}
			\pi:(\mathcal{O}/n\mathcal{O})^{\times}\to(\mathcal{O}/\gcd(n,4)\mathcal{O})^{\times},
		\end{equation}
		which is surjective. For $\alpha=0$, the isomorphism 
		\begin{equation}
			\operatorname{Gal}(K(E_1^*[\beta])/K)\cong\pi^{-1}(\operatorname{Gal}(K(E_1^*[1])/K)),
		\end{equation}
		follows from \ref{eq:(2.10)}. We also note that, $(\mathcal{O}/2^{\alpha}\beta\mathcal{O})^{\times}\cong(\mathcal{O}/2^{\alpha}\mathcal{O})^{\times}\times(\mathcal{O}/\beta\mathcal{O})^{\times}$. Then, for $\alpha\ge1$, it suffices to check if the isomorphism \ref{eq:(2.9)} holds for $n=2^{\alpha}$. We further note, that when $\alpha=1$, $\gcd(n,4)=\gcd(2,4)=2$, and the isomorphism $(2.9)$ follows trivially. Let us now consider the group homomorphism when $\alpha>1$, given by
		\begin{equation}
			\pi:(\mathcal{O}/2^{\alpha}\mathcal{O})^{\times}\to(\mathcal{O}/4\mathcal{O})^{\times}.
		\end{equation}
		We learnt in section $4$, that $\operatorname{Gal}(K(E_1^*[2^{\alpha}])/K)\cong(\mathcal{O}/2^{\alpha}\mathcal{O})^{\times}/\{\pm1\}, \, \forall\alpha\ge3$  and we have the natural projection from $(\mathcal{O}/2^{\alpha}\mathcal{O})^{\times}$ to $(\mathcal{O}/2^{\alpha}\mathcal{O})^{\times}/\{\pm1\}$. Moreover, $\pi$ induces a surjective group homomorphism between $(\mathcal{O}/2^{\alpha}\mathcal{O})^{\times}/\{\pm1\}$ to $(\mathcal{O}/4\mathcal{O})^{\times}/\{\pm1\}$. $\pi$ also induces a surjective group homomorphism from $C=\operatorname{Gal}(K(E_1^*[2^{\alpha}])/K)$ to $D=\operatorname{Gal}(K(E_1^*[4])/K)$,and the corresponding diagram of maps is commutative. Let $A=(\mathcal{O}/2^{\alpha}\mathcal{O})^{\times}$ and $B=(\mathcal{O}/4\mathcal{O})^{\times}$. Then, $\operatorname{Gal}(K(E_1^*[2^{\alpha}])/K)$ and $\pi^{-1}(\operatorname{Gal}(K(E_1^*[4])/K))$ both above have index $2$ in $(\mathcal{O}/2^{\alpha}\mathcal{O})^{\times}$. Indeed, if we consider the canonical projection $p_r:B\to B/D$, then it is surjective and the kernel of the composition $p_r\circ\pi$ is $\pi^{-1}(D)$. Then by the First Isomorphism Theorem, $A/\pi^{-1}(D)\cong B/D$. Moreover, $\operatorname{Gal}(K(E_1^*[2^{\alpha}])/K)\leq\pi^{-1}(\operatorname{Gal}(K(E_1^*[4])/K))$. Thus,
		$$\operatorname{Gal}(K(E_1^*[2^{\alpha}])/K)=\pi^{-1}(\operatorname{Gal}(K(E_1^*[4])/K)).$$
		Now, for the elliptic curve $E_2^*$, $n=2^{\alpha}3^{\beta}\gamma$, where $6\nmid\gamma$, and we assume $m_{E_2^*}=12$. When $\beta=0$, the isomorphism \ref{eq:(2.9)} follows trivially. When $\beta\ge1$, it suffices to verify the isomorphism for $n=3^{\beta}$. The verification is similar to what we did for $E_1^*$. Thus,
		$$\operatorname{Gal}(K(E_2^*[3^{\alpha}])/K)=\pi^{-1}(\operatorname{Gal}(K(E_2^*[3])/K)),$$ where $	\pi:(\mathcal{O}/3^{\alpha}\mathcal{O})^{\times}\to(\mathcal{O}/3\mathcal{O})^{\times}$ is a group homomorphism.
	\end{proof}
	\subsection{Computing the explicit Galois groups}\label{sec:5.2}
	In this section, we will use the results mentioned in sections \ref{sec:3} and \ref{sec:4} to determine the explicit groups $\operatorname{Gal}(K(E_1^*[4])/K)$ as a subgroup of $\operatorname{GL}_{2}(\mathbb{Z}/4\mathbb{Z})$ (for $K=\mathbb{Q}(\sqrt{-1}))$ and $\operatorname{Gal}(K(E_2^*[12])/K)$ as a subgroup of $\operatorname{GL}_{2}(\mathbb{Z}/12\mathbb{Z})$ (for $K=\mathbb{Q}(\sqrt{-3}))$.\\
	\\From our discussion in section \ref{sec:4}, page $12$, part $(a)$, and Theorem \ref{thm:3.8}, we know that $$\operatorname{Gal}(K(E_1^*[4])/K)\cong(\mathcal{O}/4\mathcal{O})^{\times}/\{\pm1\},$$ and from [\cite{lozano2022galois}, Theorem $9.1(3)$], we deduce that $$\operatorname{Gal}(K(E_1^*[4])/K)\leq (\mathcal{O}/4\mathcal{O})^{\times}\cong\mathbb{Z}/2\mathbb{Z}\times\mathbb{Z}/4\mathbb{Z}.$$ Then, $\operatorname{Gal}(K(E_1^*[4])/K)$ has order $4$ and index $2$ in
	$(\mathcal{O}/4\mathcal{O})^{\times}$, given by
	\begin{align*}
		(\mathcal{O}/4\mathcal{O})^{\times}=\biggl\{\begin{pmatrix}1 & 0\\ 
			0 & 1\end{pmatrix},\, \begin{pmatrix}1 & 1\\ 
			0 & 1\end{pmatrix},\, \begin{pmatrix}1 & 2\\ 
			0 & 1\end{pmatrix},\, \begin{pmatrix}1 & 3\\ 
			0 & 1\end{pmatrix},\, \begin{pmatrix}3 & 0\\ 
			0 & 3\end{pmatrix},\, \begin{pmatrix}3 & 1\\ 
			0 & 3\end{pmatrix},\, \begin{pmatrix}3 & 2\\ 
			0 & 3\end{pmatrix},\\
		\begin{pmatrix}3 & 3\\ 
			0 & 3\end{pmatrix}\biggr\}
	\end{align*}
	Moreover, any subgroup of index $2$ contains the squared elements, i.e., if $H\leq G$, and $H$ has index $2$ in $G$, then $g^2\in H$, $\forall g\in G$. After computation, we find that the squared elements of $(\mathcal{O}/4\mathcal{O})^{\times}$ are $\begin{pmatrix}1 & 0\\ 
		0 & 1\end{pmatrix},\, \begin{pmatrix}1 & 2\\ 
		0 & 1\end{pmatrix}$.\\
	\\Then using Theorem \ref{thm:3.8} in section \ref{sec:3}, for $H_f=K$ where $f=2$ and example $9.4(1)$ in \cite{lozano2022galois}, it follows that $$\operatorname{Gal}(K(E_1^*[4])/K)=\biggl<\begin{pmatrix}3 & 3\\ 
		0 & 3\end{pmatrix}\biggr>.$$
	\\Next, we want to find the group $\operatorname{Gal}(K(E_2^*[12])/K)$ as a subgroup of $\operatorname{GL}_{2}(\mathbb{Z}/12\mathbb{Z})$. From Theorem \ref{thm:4.3}, it follows that,
	\begin{equation}
		\operatorname{Gal}(K(E_2^*[4])/K)\times\operatorname{Gal}(K(E_2^*[3])/K)\cong\operatorname{Gal}(K(E_2^*[12])/K).
	\end{equation}
	We know that there exists an isomorphism of rings,
	\begin{equation*}
		\theta:\mathbb{Z}/m\mathbb{Z}\times\mathbb{Z}/n\mathbb{Z}\to\mathbb{Z}/mn\mathbb{Z},
	\end{equation*}
	such that the generator $(1,1)\in\mathbb{Z}/m\mathbb{Z}\times\mathbb{Z}/n\mathbb{Z}$ maps to $1\in\mathbb{Z}/mn\mathbb{Z}$.
	Let $A=(a_{ij})_{1\leq i,j\leq k}\in\operatorname{GL}_{k}(\mathbb{Z}/m\mathbb{Z})$ and $B=(b_{ij})_{1\leq i,j\leq k}\in\operatorname{GL}_{k}(\mathbb{Z}/n\mathbb{Z})$. Then, we define
	\begin{align*}
		\phi:\operatorname{GL}_{k}(\mathbb{Z}/m\mathbb{Z})\times\operatorname{GL}_{k}(\mathbb{Z}/n\mathbb{Z})\to\operatorname{GL}_{k}(\mathbb{Z}/mn\mathbb{Z})\\
		(A,B)\mapsto(\theta(a_{ij},b_{ij}))_{1\leq i,j\leq k}.
	\end{align*}
	Then $\phi$ is an isomorphism of groups. Next, we need to find the groups $$\operatorname{Gal}(K(E_2^*[4])/K)\leq\operatorname{GL}_{2}(\mathbb{Z}/4\mathbb{Z}),$$ and $$\operatorname{Gal}(K(E_2^*[3])/K)\leq\operatorname{GL}_{2}(\mathbb{Z}/3\mathbb{Z}).$$ From section \ref{sec:4}, we know that $$\operatorname{Gal}(K(E_2^*[4])/K)\cong(\mathcal{O}/4\mathcal{O})^{\times},$$ and $$\operatorname{Gal}(K(E_2^*[3])/K)\cong(\mathcal{O}/3\mathcal{O})^{\times}\{\pm1\}.$$ Further, $\operatorname{Gal}(K(E_2^*[4])/K)\leq(\mathcal{O}/4\mathcal{O})^{\times}$. So, $$\operatorname{Gal}(K(E_2^*[4])/K)=(\mathcal{O}/4\mathcal{O})^{\times}.$$ From [\cite{lozano2022galois}, Theorem $9.1(1)(c)$], we deduce that $$(\mathcal{O}/4\mathcal{O})^{\times}\cong\mathbb{Z}/2\mathbb{Z}\times\mathbb{Z}/2\mathbb{Z}\times\mathbb{Z}/2\mathbb{Z}.$$ So, $\operatorname{Gal}(K(E_2^*[4])/K)$ is a group of order $8$. We further deduce that, $$(\mathcal{O}/4\mathcal{O})^{\times}=\biggl\{\begin{pmatrix}a & b\\b & a\end{pmatrix}: a,b\in\mathbb{Z}/4\mathbb{Z}, \, a^2-b^2\pmod4\in\{1,3\}\biggr\}\leq\operatorname{GL}_{2}(\mathbb{Z}/4\mathbb{Z}).$$
	Then, we obtain,
	\begin{align*}
		\operatorname{Gal}(K(E_2^*[4])/K)=\biggl\{\begin{pmatrix}1 & 0\\ 
			0 & 1\end{pmatrix},\, \begin{pmatrix}1 & 2\\ 
			2 & 1\end{pmatrix},\, \begin{pmatrix}2 & 1\\ 
			2 & 1\end{pmatrix},\, \begin{pmatrix}2 & 3\\ 
			3 & 2\end{pmatrix},\, \begin{pmatrix}3 & 0\\ 
			0 & 3\end{pmatrix},\, \begin{pmatrix}3 & 2\\ 
			2 & 3\end{pmatrix},\, \begin{pmatrix}0 & 1\\ 
			1 & 0\end{pmatrix},\\
		\begin{pmatrix}0 & 3\\ 
			3 & 0\end{pmatrix}\biggr\}.
	\end{align*}
	Now, the group $\operatorname{Gal}(K(E_2^*[3])/K)$ has index $2$ in $(\mathcal{O}/3\mathcal{O})^{\times}$. From [\cite{lozano2022galois}, Lemma $8.2(1)$], it follows that $(\mathcal{O}/3\mathcal{O})^{\times}\cong\mathbb{Z}/6\mathbb{Z}$, and 
	$$(\mathcal{O}/3\mathcal{O})^{\times}=\biggl\{\begin{pmatrix}a & b\\0 & a\end{pmatrix},\, a\in\{1,2\},\, b\in\{0,1,2\}\biggr\}\leq\operatorname{GL}_{2}(\mathbb{Z}/3\mathbb{Z}).$$ If we compute the squares of the elements of the group above then we obtain,
	$$\operatorname{Gal}(K(E_2^*[3])/K)=\biggl<\begin{pmatrix}1 & 1\\ 
		0 & 1\end{pmatrix}\biggr>.$$
	One can verify this by observing that, $\operatorname{Gal}(K(E_2^*[3])/K)=G_{E_2^*,K,3}=\rho_{E_2^*,3}(\operatorname{Gal}(\overline{H_f}/H_f))$, where $H_f=K$ and $f=2$ in Theorem $3.7$. Finally, we need to compute the group $\operatorname{Gal}(K(E_2^*[12])/K)$. In fact, we only need to know the diagonal entries to compute the trace of the matrices in $\operatorname{Gal}(K(E_2^*[12])/K)$. So, the diagonal entries are $\theta(0,1)=4,\theta(1,1)=1,\theta(2,1)=10,\theta(3,1)=7$. Then, we obtain
	$|(\operatorname{Gal}(K(E_1^*[4])/K))_r|=|\operatorname{Gal}(K(E_1^*[4])/K)|$ for any $r\equiv 2\pmod4$ and $|(\operatorname{Gal}(K(E_1^*[4])/K))_r|=0$, otherwise. Similarly, we obtain $|(\operatorname{Gal}(K(E_2^*[12])/K))_r|=\frac{1}{2}|\operatorname{Gal}(K(E_2^*[12])/K)|$ for any $r\equiv 2,8\pmod{12}$ and\\ $|(\operatorname{Gal}(K(E_2^*[12])/K))_r|=0$, otherwise. Then from \ref{eq:(2.12)} for primes $p$, we obtain
	\begin{equation}\label{eq:(5.5)}
		C_{E_1^*,r}=\begin{cases}
			\frac{2}{\pi}\prod\limits_{\substack{p\nmid2\\ p\mid r}}\frac{p}{p-\left(\frac{-1}{p}\right)}\prod\limits_{\substack{p\nmid2\\p\nmid r}}\left(1-\frac{\left(\frac{-1}{p}\right)}{(p-1)\left(p-\left(\frac{-1}{p}\right)\right)}\right), & r\equiv2\pmod4\\
			0 & \text{otherwise,}
		\end{cases}
	\end{equation}
	and
	\begin{equation}\label{eq:(5.6)}
		C_{E_2^*,r}=\begin{cases}
			\frac{6}{\pi}\prod\limits_{\substack{p\nmid6\\ p\mid r}}\frac{p}{p-\left(\frac{-3}{p}\right)}\prod\limits_{\substack{p\nmid6\\p\nmid r}}\left(1-\frac{\left(\frac{-3}{p}\right)}{(p-1)\left(p-\left(\frac{-3}{p}\right)\right)}\right), & r\equiv2,8\pmod{12}\\
			0 & \text{otherwise.}
		\end{cases}
	\end{equation}
	where $\left(\frac{.}{p}\right)$ denotes the Legendre symbol for an odd prime $p$. Moreover, from Theorem \ref{thm:2.1} and Theorem \ref{thm:2.2}, we obtain
	\begin{equation}\label{eq:(5.7)}
		\overline{\omega}_{E_1,r}=\begin{cases}
			\frac{1}{2}\prod\limits_{p\nmid2r}\left(1-\frac{\left(\frac{-1}{p}\right)}{p-1}\right), & r\equiv2\pmod{4}\\
			0 & \text{otherwise.}
		\end{cases}
	\end{equation}
	and
	\begin{equation}\label{eq:(5.8)}
		\overline{\omega}_{E_2,r}=\begin{cases}
			\frac{\sqrt{3}}{2}\prod\limits_{p\nmid2r}\left(1-\frac{\left(\frac{-3}{p}\right)}{p-1}\right), & r\equiv2\pmod{6}\\
			0 & \text{otherwise.}
		\end{cases},
	\end{equation}
	respectively. A natural question is whether this infinite products in \ref{eq:(5.7)} and \ref{eq:(5.8)} converge. We observe that the infinite product over primes can be written as $\prod\limits_{p}\left(1-\frac{\chi(p)}{p-1}\right)=\prod\limits_{p}\left(1-\frac{\chi(p)}{p}\right)\left(\frac{1-\frac{\chi(p)}{p-1}}{1-\frac{\chi(p)}{p}}\right)=\frac{1}{L(1,\chi)}\prod\limits_{p}\left(\frac{1-\frac{\chi(p)}{p-1}}{1-\frac{\chi(p)}{p}}\right)$, where $\chi$ is a non-principal Dirichlet character. This is true because the Euler product of $L(1,\chi)$ defined as $(1-\chi(p)p^{-1})^{-1}$ is convergent after writing the product in increasing order of primes and the terms of other infinite product on the right hand side is of the form $1+O\left(\frac{1}{p^2}\right)$, hence it is convergent as well. So, it remains to be verified if $C_{E_1^*,r}=\omega_{E_1,r}$ and $C_{E_2^*,r}=\omega_{E_2,r}$.
	To this end, we observe, that other than $r\equiv2\pmod4$, $C_{E_1^*,r}=0=\omega_{E_1,r}$, and other than $r\equiv2,8\pmod{12}$, $C_{E_2^*,r}=0=\omega_{E_2,r}$. Then, we only need to verify the equality when $r\equiv2\pmod4$ for $E_1,E_1^*$ and $r\equiv2,8\pmod{12}$ for $E_2,E_2^*$. So, we observe that for any $r\equiv2\pmod4$, we can rewrite $C_{E_1^*,r}$ in \ref{eq:(5.5)} as
	\begin{align*}
		C_{E_1^*,r}=\frac{2}{\pi}\prod\limits_{\substack{p\nmid2\\ p\mid r}}\frac{p}{p-\left(\frac{-1}{p}\right)}\prod\limits_{\substack{p\nmid2\\p\nmid r}}\left(1-\frac{\left(\frac{-1}{p}\right)}{(p-1)\left(p-\left(\frac{-1}{p}\right)\right)}\right)\\
		=\frac{2}{\pi}L(1,\chi_4)\prod\limits_{\substack{p\nmid2\\p\nmid r}}\frac{p-\left(\frac{-1}{p}\right)}{p}\prod\limits_{\substack{p\nmid2\\p\nmid r}}\left(1-\frac{\left(\frac{-1}{p}\right)}{(p-1)\left(p-\left(\frac{-1}{p}\right)\right)}\right)\\
		=\frac{2}{\pi}L(1,\chi_4)\prod\limits_{p\nmid2r}\left(1-\frac{\left(\frac{-1}{p}\right)}{p-1}\right)\\
		=\frac{2}{\pi}L(1,\chi_4)2\overline{\omega}_{E_1,r}\\
		=\frac{2}{\pi}\frac{\pi}{4}2\overline{\omega}_{E_1,r}\\
		=\overline{\omega}_{E_1,r},
	\end{align*}
	where $L(1,\chi_4)=\prod\limits_{p}\left(1-\frac{\chi_4}{p}\right)^{-1}=\frac{\pi}{4}$ is the Dirichlet L-function $(L(s,\chi_4))$ of non-principal character modulo $4$,  $\chi_4$ at $s=1$. Similarly, for any $r\equiv2,8\pmod{12}$, we can rewrite $C_{E_2^*,r}$ in \ref{eq:(5.6)} as
	\begin{align*}
		C_{E_2^*,r}=\frac{6}{\pi}\frac{1}{2}\prod\limits_{\substack{p\nmid6\\ p\mid r}}\frac{p}{p-\left(\frac{-3}{p}\right)}\prod\limits_{\substack{p\nmid6\\p\nmid r}}\left(1-\frac{\left(\frac{-3}{p}\right)}{(p-1)\left(p-\left(\frac{-3}{p}\right)\right)}\right)\\
		=\frac{3}{\pi}\frac{3}{2}L(1,\chi_3)\prod\limits_{\substack{p\nmid6\\ p\nmid r}}\frac{p-\left(\frac{-3}{p}\right)}{p}\prod\limits_{\substack{p\nmid6\\p\nmid r}}\left(1-\frac{\left(\frac{-3}{p}\right)}{(p-1)\left(p-\left(\frac{-3}{p}\right)\right)}\right)\\
		=\frac{3}{\pi}\frac{3}{2}L(1,\chi_3)\prod\limits_{p\nmid2r}\left(1-\frac{\left(\frac{-3}{p}\right)}{p-1}\right)\\
		=\frac{3}{\pi}\frac{3}{2}L(1,\chi_3)\frac{2}{\sqrt{3}}\overline{\omega}_{E_2,r}\\
		=\frac{3}{\pi}\frac{3}{2}\frac{\pi}{3\sqrt{3}}\frac{2}{\sqrt{3}}\overline{\omega}_{E_2,r}\\
		=\overline{\omega}_{E_2,r},
	\end{align*}
	where $L(1,\chi_3)=\prod\limits_{p}\left(1-\frac{\chi_3}{p}\right)^{-1}=\frac{\pi}{3\sqrt{3}}$ is the Dirichlet L-function  of non-principal character modulo $3$, $\chi_3$ at $s=1$.\\
	\\
	Let us define the following:\\
	\\
	i) $\Gamma:=\{r\in\mathbb{Z}\mid\operatorname{tr} a\equiv r\pmod{m_E}, \, a\in\operatorname{Gal}(K(E[m_E])/K)\leq\operatorname{GL}_{2}(\mathbb{Z}/m_E\mathbb{Z})\}$.\\
	\\ii) $\kappa_{\text{odd}}:=\frac{|(\operatorname{Gal}(K(E[m_E])/K))_r|}{|\operatorname{Gal}(K(E[m_E])/K)|}$ for odd $r\in\Gamma$, and $\kappa_{\text{even}}:=\frac{|(\operatorname{Gal}(K(E[m_E])/K))_r|}{|\operatorname{Gal}(K(E[m_E])/K)|}$ for even $r\in\Gamma$.\\
	\\
	One can observe that for all the elliptic curves mentioned below in Table $1$, $g=1$ in \ref{eq:(2.5)}. So, the corresponding $\lambda,\mu$ are $0$ and $g_1=1$. Moreover, $\Omega_2(D;1,r)=1$ and $\left(\frac{2r}{D}\right)=-1$ for $r\notin\Gamma$, $D\nmid r$, and  $\left(\frac{2r}{D}\right)=0$ for $D\mid r$. We would also like to note that for a non-prinicipal Dirichlet character modulo $D$ taking values $\pm1$, $L(1,\chi)=\frac{\pi}{\sqrt{D}}$, for the values of $D$ in Table $1$.\\
	\\ 
	Then for $r\in\Gamma$, we replicate the methods as above to obtain the following table:
		\begin{table}[!h]
			\begin{center}
				\begin{tabular}{|c|c|c|c|c|c|c|c|c|} 
		\hline
		$E$ & $D$ & $m_E$ & $\kappa_{\text{odd}}$ & $\kappa_{\text{even}}$ & $\left(\frac{2r}{D}\right)$ & $\xi(D,r)$, $2\nmid r$ & $\xi(D,r)$, $2\mid r$  & $\xi_D(1,r)$ \\
		\hline
		$E_3$  & $7$  & $28$ & $0$ & $\frac{1}{6}$ & $1$ & $0$ & $1$ & $1$ \\
		\hline
		$E_4$  & $11$  & $44$ & $\frac{1}{15}$ & $\frac{1}{30}$ & $1$ & $2$ & $1$ & $1$ \\
		\hline
		$E_5$ & $19$ & $76$ & $\frac{1}{27}$ & $\frac{1}{54}$ & $1$ & $2$ & $1$ & $1$ \\
		\hline
		$E_6$ & $43$ & $172$ & $\frac{1}{63}$ & $\frac{1}{126}$ & $1$ & $2$ & $1$ & $1$ \\
		\hline
		$E_7$ & $67$ & $268$ & $\frac{1}{99}$ & $\frac{1}{198}$ & $1$ & $2$ & $1$ & $1$ \\
		\hline
		$E_8$ & $163$ & $652$ & $\frac{1}{243}$ & $\frac{1}{486}$ & $1$ & $2$ & $1$ & $1$ \\
		\hline
	\end{tabular}
	\caption*{Table 1}
    \end{center}
    \end{table}
	\\Then, for each of the CM elliptic curves $E$ in Table $1$, plugging in the corresponding values in the formula of $C_{E,r}$ and $\overline{\omega}_{E,r}$, we obtain the desired equality of the two constants. 
	\section{Concluding Remarks}
	We have thus shown that the constant $C_{E,r}$ in the asymptotics of the Lang-Trotter Conjecture \ref{con:1.1} is equal to the constant $\overline{\omega}_{E,r}$ in Theorem \ref{thm:1.4} for the $20$ CM elliptic curves as mentioned in section \ref{sec:1}. Thus confirming Conjecture \ref{con:1.5} for these isogeny classes of elliptic curves. Moreover, it proves that for these CM elliptic curves the Lang-Trotter Conjecture \ref{con:1.1} is equivalent to the Hardy-Littlewood Conjecture \ref{con:1.2} which extends the work of Qin \cite{qin2021} where the same has been proved for CM elliptic curves of the form $y^2=x^3+Dx$, $D\in\mathbb{Z}$. This helps us to compute the much harder Lang-Trotter constant in terms of the explicit constant given in \cite{wan2021langtrotter}, at least for the elliptic curves as discussed above. A possible extension of this work would significantly depend on finding a description of the Galois group $\operatorname{Gal}(F(E[S^{\infty}])/F)$ in Theorem \ref{thm:4.1} as a subgroup of a general linear group $\operatorname{GL}_{2}(\mathbb{Z}/n\mathbb{Z})$, where the $n$ depends on $m_E$ as its divisor. For this, one can try to work on the description of the Galois groups given by the $m_E$-division field of the CM elliptic curves obtained by twisting the elliptic curves in Table 1 \cite{campagna2022entanglement} as mentioned in section $6$\cite{campagna2022entanglement}, as a subgroup of $\operatorname{GL}_{2}(\mathbb{Z}/m_E\mathbb{Z})$. Lastly, for the CM elliptic curves with CM discriminant $-8$ as mentioned in Table $1$\cite{campagna2022entanglement}, it is not possible to figure out which among $J_1^{\prime}$ and $J_2^{\prime}$ in Theorem \ref{thm:3.8}$(2)$ would be the image of $\rho_{E,8}(\operatorname{Gal}(\overline{\mathbb{Q}}/K))$. But one should be able to show the equality of the two constants using some computational tools as developed by \cite{rouse_sutherland_zureick-brown_2022} to compute the Galois group $\operatorname{Gal}(K(E[8])/K)$ and replicating our methods as shown in section \ref{sec:5.2}. This would essentially prove Conjecture \ref{con:1.5} for $24$ CM elliptic curves.
	\bibliographystyle{apa}
	\bibliography{References}
	\vspace{0.2in}
	\texttt{Anish Ray, Department of Mathematics and Computer Science, University of M\"{u}nster, Germany}\\
	\texttt{Email id: }\url{anishray49@gmail.com}
\end{document}